%% file: Norm_Minimization_Problem-Revision.tex
\documentclass[]{interact}

\usepackage{epstopdf}
\usepackage[caption=false]{subfig}
\usepackage{graphicx}
\usepackage{float}
\usepackage[numbers,sort&compress]{natbib}
\bibpunct[, ]{[}{]}{,}{n}{,}{,}
\makeatletter
\def\NAT@def@citea{\def\@citea{\NAT@separator}}
\makeatother
\usepackage{hyperref}
\hypersetup{
	colorlinks=true,
	linkcolor=blue, 
	citecolor=red, 
	urlcolor=blue  } 
\usepackage{mathptmx}   
\usepackage{enumitem}
\theoremstyle{plain}
\newtheorem{theorem}{Theorem}[section]
\newtheorem{lemma}[theorem]{Lemma}

\newtheorem{proposition}[theorem]{Proposition}

\theoremstyle{definition}

\newtheorem{example}[theorem]{Example}

\theoremstyle{remark}
\newtheorem{remark}[theorem]{Remark}

\renewcommand {\theenumi} {\rm\roman{enumi}}

\usepackage{color}
\usepackage{comment}
\usepackage[colorinlistoftodos]{todonotes}
\input{macro}

\begin{document}

\title{
Dual characterizations of norm  minimization problems}

\author{
\name{Nguyen Duy Cuong
}
\thanks{CONTACT Nguyen Duy Cuong. Email: ndcuong@ctu.edu.vn}
\thanks{{Dedicated to Professor Phan Quoc Khanh on the occasion of his 80th birthday}}
\affil{
Faculty of Mathematics, College of Natural Sciences, Can Tho University, Can Tho {City}, Vietnam
}
}

\maketitle

\begin{abstract}
The paper studies a general norm minimization problem on a product of normed vector spaces.
We establish dual necessary and sufficient optimality conditions and derive explicit formulas for the corresponding solution sets.
These formulas are obtained under the assumption that one optimal solution together with its associated dual vectors arising from the optimality conditions is known.
Three important cases of product norms, namely the sum norm, maximum norm and  $p$-norm, are also studied.
Several examples in  finite and infinite dimensional spaces equipped with various types of norms are presented to illustrate the established results.
\end{abstract}


\begin{amscode}
49J52; 49J53; 49K40; 90C30; 90C46
\end{amscode}

\setcounter{tocdepth}{2}

\section{Introduction}\label{S1}
Given a finite set of points in a normed vector space $X$, the  Fermat–Torricelli problem seeks a point minimizing the sum of the distances to these points, while the Chebyshev centre problem aims to find a point minimizing the maximum of the distances. 
Both problems have been intensively studied and have diverse applications in many areas such as location theory, data analysis,  approximation theory and related fields \cite{MarSwaWei02,ReiTru23,BolMarSol99, DurMic85,AliTsa19,Ves97,MorNam11,KazLiuMor25}.

There are several important points around the Fermat--Torricelli  and Chebyshev centre problems.
Computational methods have been studied in finite dimensional spaces equipped with the Euclidean norm \cite{MorNam19,MorNam11} and, more generally, in Hilbert spaces \cite{BalMax22,ReiTru23}.
In infinite-dimensional settings, much attention has been devoted to properties of the solution set and the associated solution mapping such as  existence, uniqueness, compactness, continuity and stability \cite{Ves97,AmiMac84,AliTsa19}.
Some extensions of the classical models have also been investigated, involving, e.g.,  the replacing the given points by  sets \cite{AliTsa19,MorNamSal12,MorNam11} or  weighted formulations of the problems \cite{ZacZou08,Pau25,BanDut03,Ale14}.
Much work targets  dual optimality conditions, primarily for the Fermat--Torricelli problem and its extensions \cite{DurMic85,MorNam11,MorNamSal12,MarSwaWei02}.
Finite-dimensional models with non-Euclidean norms \cite{MarSwaWei02} or more general gauge-type functions \cite{DurMic85} have also been investigated.

The aforementioned problems are typically studied separately in the literature.
The primary reason for this lies in their seemingly different formulations:
the Fermat–Torricelli problem uses the sum norm on $X^n$, whereas Chebyshev centre problem uses the maximum norm.
From the theoretical point of view, however, this distinction does not seem essential and it  could make sense investigating a general norm minimization problem with an arbitrary norm on $X^n$.

Some authors have studied unified modeling frameworks encompassing the Fermat--Torricelli and  Chebyshev centre problems.
For instance, in \cite{Ves97,BanDut03}, among other things,  structural properties of the corresponding solution sets and maps are studied, whereas in \cite{Pau25} the authors establish connections between the Fermat--Torricelli and  Chebyshev centre problems via the concept of Birkhoff--James orthogonality.
To the best of our knowledge, however, there is  no general framework that systematically focuses on dual optimality conditions and explicit solution constructions in which the Fermat--Torricelli and Chebyshev centre problems arise as particular cases.

We note that there are two distinct problems commonly referred to as Chebyshev centre problems.
The first one, discussed above, concerns a finite collection of points and seeks a centre that minimizes the maximum of the distances
 to these points.
 {This problem is also known in the literature as the Sylvester problem \cite{MorNam14,NamHoa13,NamHoaAn14}.}
The second one is the problem of finding a largest inscribed ball in a given set \cite{BotTur94,BaeCas23}.
While the former can be viewed as a min–max covering problem with a free centre, the latter is a max–min packing problem in which the centre is constrained to lie in the set.
Although these two problems are closely related in certain geometric settings, they are not equivalent in general.
In this paper, we focus exclusively on the first problem as a special case of our general framework; the second problem lies beyond the scope of the present work.

We have encountered the general norm minimization problem in our recent study of generalized separation results for a finite collection of  sets \cite{CuoKru25}.
Such results play an important role in establishing optimality conditions \cite{CuoKru} and calculus rules for normal cones, subdifferentials, and coderivatives \cite{Mor06.1}.
In order to establish  dual separation results, it is necessary to impose certain inequalities involving norms on a product of vector spaces.
Although such estimates arise naturally and fit well within our framework, the sharpest bounds have not been obtained since the minimal values of the product norms involved are  unknown.
Sharper estimates allow for weaker assumptions and, consequently, strengthen the resulting statements.

In this paper, we develop a unified framework in which dual optimality conditions and solution set properties and constructions of norm minimization problems such as the Fermat--Torricelli and  Chebyshev centre problems  can be derived in a unified manner.
The analysis carried out in this paper, among other contributions, provides a deeper understanding of norm minimization problems and the corresponding generalized separation results.

We would like to emphasize that two norms are essential for our model.
The first is the norm on the underlying space $X$, while the second is the norm on  the product space $X^n$.
The norm on $X$ determines the level of analytical and computational complexity of the problem.
For instance, when $X$ is a finite dimensional space equipped with the Euclidean norm or a Hilbert space, numerical solution methods can be effectively employed due to the smoothness of the  norm away from the origin.
In contrast, when nonsmooth norms such as the maximum norm are involved, standard gradient-based numerical algorithms are no longer applicable.
On the other hand,  the norm on $X^n$ governs how the individual norms on $X$ interact:
the sum norm gives rise to the Fermat--Torricelli problem; the maximum norm leads to the Chebyshev centre problem;  the $p$-norm generalizes the $p$-Fermat--Torricelli problem.

Although the sum norm and maximum norms can be viewed as particular cases of $p$-norms corresponding to $p=1$ and $p=\infty$, respectively, the analytical techniques required to handle these norms differ substantially.
More precisely, for the sum norm, the associated optimality conditions can be derived by applying the standard convex subdifferential sum rule. In contrast, for the maximum norm, this sum rule is no longer applicable in general; instead, one must employ appropriate subdifferential chain rules for max-type functions.
Finally, for $p$-norms with $1<p<\infty$, the norm is smooth away from the origin whenever the  norm on $X$ is smooth, and consequently,  classical differential calculus can be applied to derive optimality conditions.

In the present paper, we investigate a general norm minimization problem and its particular cases by employing standard tools from convex  and functional analysis.
We first study a general model, focusing on fundamental properties of the solution set such as existence, uniqueness and compactness, and establish dual necessary and sufficient optimality conditions.
We then examine three important norms on the product space $X^n$, namely the sum, maximum and  $p$-norms, and provide comparisons with existing results in the literature whenever applicable.
Several examples are presented to illustrate the established results in  finite and infinite dimensional spaces.
In particular, when $X := \mathbb{R}^n$, we consider the cases of the sum, maximum and $p$-norms on $X$.
These examples illustrate how the choice of the norm on $X$ affects the structure of the solution sets.

In some situations, identifying the entire solution set of an optimization problem is very useful. 
For example, in a facility location problem minimizing total distance to a set of points, practical constraints such as location availability or regulations may prevent using certain points.
 Knowing all optimal points allows selecting a feasible location without increasing the minimal cost. 
The study of characterizations of solution sets of convex programs can be found in \cite{Man88,Bur91}.
Of course, finding all solutions of a general norm minimization problem is a challenging task.
Nevertheless, under some reasonable assumptions, such a task can be carried out theoretically.
In this paper, following the ideas developed in \cite{MarSwaWei02,DurMic85}, we show that explicit descriptions of the entire solution set can be obtained once some optimal solution is known together with the associated dual vectors arising from the  dual optimality conditions.
Several examples are provided to illustrate the construction of solution sets.

One of the main tools in our study is a product norm construction scheme studied in \cite{SaiKatTak00} and recently refined in \cite{Cuo25}.
Starting from a continuous convex function satisfying some reasonable conditions, we obtain a broad class of norms on product spaces including the sum, maximum and  $p$-norms.
This class of norms provides a convenient setting for our analysis since explicit formulas for the corresponding dual norm can be derived.
More precisely, when a primal norm is generated by a continuous convex function, its dual norm is determined by a conjugate-type convex function associated with the primal one.
The duality relation between these generating functions then translates directly into precise relationships between the primal and dual norms.

The next Section~\ref{S1+} collects general definitions and preliminary results that are used throughout the paper.
Section~\ref{S2} is devoted to the study of a class of general norms on  a product of normed spaces constructed via a family of convex functions defined on the standard simplex of $\mathbb{R}^n$ and establishes technical results needed for the subsequent analysis.
Section~\ref{S3} studies a general norm minimization problem with particular emphasis on dual necessary and sufficient optimality conditions and solution set constructions.
Fundamental properties of the solution set such as existence, uniqueness and compactness are also investigated.
Sections~\ref{S4}--\ref{S6} examine, respectively, the three particular cases of the general model corresponding to the sum, maximum and $p$-norms on $X^n$.
For each case, several examples in  finite and infinite dimensional spaces equipped with various types of norms are provided.

\section{Preliminaries}\label{S1+}
Our basic notation is standard, see, e.g., \cite{Mor06.1,RocWet98} .
Throughout the paper, if not explicitly stated otherwise, $X$ is a vector space equipped with a norm $\|\cdot\|$.
{We denote by $0_X$ the zero vector in $X$.}
Given an integer $n\ge 2$, the product space $X^n$ is assumed to be endowed with a norm $\vertiii{\cdot}$.
The topological dual of a normed space $X$ is denoted by $X^*$, while $\langle\cdot,\cdot\rangle$ denotes the bilinear form defining the pairing between the two spaces.
The dual norms of $\|\cdot\|$ and $\vertiii{\cdot}$ are denoted by $\|\cdot\|^*$ and $\vertiii{\cdot}^*$, {respectively}.
Both the primal and dual norms are usually labelled by subscripts indicating the underlying space.
When $X$ is a Hilbert space 
with an inner product $\langle\cdot,\cdot\rangle$, the associated norm is given by  
$\|\cdot\|:=\sqrt{\ang{\cdot,\cdot}}$.
Symbols $\R$ and $\R_+$ denote the real line equipped with the standard norm and the set of all nonnegative real numbers, respectively.
We write $\infty$ instead of $+\infty$ and employ the convention that $\frac{1}{\infty}=0$ and $\frac{1}{0}=\infty$.
Throughout the paper, we assume that $p,q\in[1,\infty]$ satisfy the condition
$\frac{1}{p}+\frac{1}{q}=1$.
The $p$-norm on 
$X^n$ is defined by  
\begin{equation}\label{pnorm}
\vertiii{x}_{p}
:= \begin{cases}
\left(\|x_1\|^p+\ldots+\|x_n\|^p\right)^{\frac{1}{p}}  & \text{if } p\in[1,\infty),\\	
\max\{\|x_1\|,\ldots,\|x_n\|\} & \text{if } p=\infty
\end{cases} 
\end{equation}
for all $x:=(x_1,\ldots,x_n)\in X^n$.
It is well known that
\begin{equation*}
\vertiii{x^*}^*_{p}=\vertiii{x^*}_q
= \begin{cases}
\max\{\|x^*_1\|^*,\ldots,\|x^*_n\|^*\} & \text{if } p=1,\\
\left(\|x^*_1\|^{*q}+\ldots+\|x^*_n\|^{*q}\right)^{\frac{1}{q}}  & \text{if } p\in(1,\infty]
\end{cases} 
\end{equation*}
for all $x^*:=(x^*_1,\ldots,x^*_n)\in (X^*)^n$, where $\vertiii{\cdot}^{*}_{p}$ denotes the dual norm of  $\vertiii{\cdot}_{p}$, and $\vertiii{\cdot}_{q}$ is the $q$-norm on the product space $(X^*)^n$.

For a convex function $f:X\to\R\cup\{\infty\}$ on a normed space $X$, its domain is defined  by
$\dom f:=\{x \in X\mid f(x) <\infty\}$.
The convex subdifferential of $f$ at $\bar x\in\dom f$
is defined by
\begin{gather*}
\sd f(\bar x):= \left\{x^* \in X^*\mid \ang{x^*,x-\bx}\le f(x)-f(\bx)\;\;\text{for all}\;\;x\in X\right\}.
\end{gather*}
We set $\partial{f}(\bx):=\emptyset$ if $\bx\notin\dom f$.
For a norm $\|\cdot\|$ on a vector space $X$,
it is well known \cite[Example~3.36]{MorNam22} that
\begin{equation}\label{sn}
\partial\|\cdot\|(x) = 
\begin{cases}
\left\{x^*\in X^*\mid \|x^*\|^* =1\;\;\text{and}\;\;\langle x^*,x\rangle=\|x\|\right\} & \text{if } x \ne 0_X,\\
\{x^*\in X^*\mid \|x^*\|^* \le 1\} & \text{otherwise.}
\end{cases}
\end{equation}

A convex function $f$ is  {strictly convex} if $f(\lambda x+(1-\lambda)y)<\lambda f(x)+(1-\lambda)f(y)$
for any distinct $x,y\in\dom f$ and $\lambda\in(0,1)$.
A norm $\|\cdot\|$ is  {strictly convex }if $\norm{{x+y}}<\|x\|+\|y\|$ for any nonzero  $x,y\in X$ such that $x\ne\al y$ for all $\al>0$.
The sign function is defined by
$\text{sgn}(x):=1$ if $x>0$, $\text{sgn}(x):=0$ if $x=0$, and
$\text{sgn}(x):=-1$ if $x<0$.
The convex hull of vectors $v_1,\ldots,v_n\in X$ is denoted by $\text{conv}\{v_1,\ldots,v_n\}$.
\if{
Let $X:=\R^n$ and an integer $k\in[1,n]$.
Define
\begin{gather}\label{k-norm}
\|x\|_{(k)}:=\max_{1\le i_1<\ldots< i_n\le n}\{|x_{i_1}|+\ldots+|x_{i_k}|\}
\end{gather}	
for all $x:=(x_1,\ldots,x_n)\in\R^n$.
Then $\|\cdot\|_{(k)}$ is a norm on $\R^n$.
It is called a polyhedral norm.
It is clear that
\begin{gather}\label{1.1}
\|\cdot\|_\infty=\|\cdot\|_{(1)}\le \ldots\le\|\cdot\|_{(k)}\le\ldots\le \|\cdot\|_{(n)}=\|\cdot\|_1.	
\end{gather}	
The dual norm of \eqref{k-norm} is given by
\begin{gather}\label{dk-norm}
\|\cdot\|^*_{(k)}:=\max\left\{\dfrac{\|\cdot\|_1}{k},\|\cdot\|_\infty\right\}.
\end{gather}	
It follows from \eqref{1.1} that
\begin{gather*}
\|\cdot\|_\infty=\|\cdot\|^*_1=\|\cdot\|^*_{(n)}\le
\ldots\le\|\cdot\|^*_{(k)}\le\ldots\le\|\cdot\|^*_{(1)}=\|\cdot\|^*_\infty=\|\cdot\|_1.
\end{gather*}	
}\fi

\section{Product Norm Constructions}\label{S2}
The following general norm construction has been developed in \cite{SaiKatTak00} and recently refined in  \cite{Cuo25}.
Consider a convex and compact set $\Omega_n:=\{(t_1,\ldots,t_{n})\in{\R^{n}_+}\mid \sum_{i=1}^nt_i =1\}.$
It is clear that $\max\{t_1,\ldots,t_n\}\ge \frac{1}{n}$ for all $(t_1,\ldots,t_{n})\in\Omega_n$.
Let $\pmb{\Psi}_n$ denote the class of all convex continuous  functions $\psi:\Omega_n\to\R$ such that
$\psi(\mathbf{e}_1)=\cdots=\psi(\mathbf{e}_n)=1$, where $\mathbf{e}_1,\ldots,\mathbf{e}_n$ are the standard basis vectors of $\R^n$, and
\begin{gather}\label{psi}
\psi(t_1,\ldots,t_n)\ge (1-t_i)\cdot\psi\left(\dfrac{t_1}{1-t_i},\ldots,\dfrac{t_{i-1}}{1-t_i},0,\dfrac{t_{i+1}}{1-t_i},\ldots,\dfrac{t_{n}}{1-t_i}\right)
\end{gather}	
for all $(t_1,\ldots,t_n)\in\Omega_n$ with $t_i<1$ 
$(i=1,\ldots,n)$.
A function $\psi\in \pmb{\Psi}_n$ is called \textit{permutation-symmetric} if $\psi(t_1,\ldots,t_n)=\psi(t_{\sigma(1)},\ldots,t_{\sigma(n)})$ for all $(t_1,\ldots,t_n)\in\Omega_n$ and  every permutation
$\sigma$ of $\{1,\ldots,n\}$.

\begin{example}
The functions
\begin{gather}\label{ppsi}
\psi_p(t_1,\ldots,t_n)
:=\begin{cases}
\left(t_1^p+\cdots+t_n^p\right)^{\frac{1}{p}}  & \text{if } p\in[1,\infty),\\
\max\{t_1,\ldots,t_n\} & \text{if } p=\infty
\end{cases} 
\end{gather}
for all $(t_1,\ldots,t_n)\in\Omega_n$ belong to $\pmb{\Psi}_n$ and are  permutation-symmetric.
Note that 
\begin{gather}\label{E3.1-1}
\psi_1(t_1,\ldots,t_n)=1\;\;\text{for all}\;\;(t_1,\ldots,t_n)\in\Omega_n.
\end{gather}	
\begin{proof}
It suffices to verify condition~\eqref{psi} since the other  conditions are obviously satisfied.
Let $(t_1,\ldots,t_n)\in\Omega_n$ and fix $i\in\{1,\ldots,n\}$ with $t_i<1$.
If $p\in[1,\infty)$, then
\begin{align*}
\psi_p(t_1,\ldots,t_n)
&=\left(\sum_{j=1}^{n}t^p_j\right)^{\frac{1}{p}}
\ge \left(\sum_{j\ne i}^{}t^p_j\right)^{\frac{1}{p}}
=(1-t_i)\left(\sum_{j\ne i}^{}\left(\frac{t_j}{1-t_i}\right)^p \right)^{\frac{1}{p}}\\
&=(1-t_i)\cdot\psi_p\left(\dfrac{t_1}{1-t_i},\ldots,\dfrac{t_{i-1}}{1-t_i},0,\dfrac{t_{i+1}}{1-t_i},\ldots,\dfrac{t_{n}}{1-t_i}\right).
\end{align*}	
If $p=\infty$, then
\begin{align*}
\psi_\infty(t_1,\ldots,t_n)
&=\max\{t_1,\ldots,t_n\}
\ge \max_{j\ne i}t_j=(1-t_i)\max_{j\ne i}\frac{t_j}{1-t_i}\\
&=(1-t_i)\cdot\psi_\infty\left(\dfrac{t_1}{1-t_i},\ldots,\dfrac{t_{i-1}}{1-t_i},0,\dfrac{t_{i+1}}{1-t_i},\ldots,\dfrac{t_{n}}{1-t_i}\right).
\end{align*}	
This completes the proof.
\end{proof}	
\end{example}

The next statement collects some basic facts needed for the subsequent analysis.
\begin{proposition}\label{P2.2}
Let $\psi\in\pmb{\Psi}_n$.
The following assertions hold:
\begin{enumerate}
\item\label{P2.2-1}
$\max\{t_1,\ldots,t_n\}\le\psi(t)\le 1$ for all $t:=(t_1,\ldots,t_n)\in\Omega_n$;
\item\label{P2.2-2}
if $\psi$ is  \textit{permutation-symmetric}, then it attains the minimum at $(\frac{1}{n},\ldots,\frac{1}{n})$.
\end{enumerate}
\end{proposition}	

\begin{proof}
Assertion \eqref{P2.2-1} is proved in \cite{Cuo25}.
We now prove assertion \eqref{P2.2-2}.
Let $(t_1,\ldots,t_n)\in\Omega_n$.
By the permutation-symmetry and convexity of $\psi$, we have
\begin{align*}
\psi\left(\frac{1}{n},\ldots,\frac{1}{n}\right)
&=\psi \left(\frac{t_1+\ldots+t_n}{n},\ldots,\frac{t_1+\ldots+t_n}{n}\right)\\
&=\psi\left(\dfrac{1}{n}(t_1,t_2,\ldots,t_n)+\dfrac{1}{n}(t_2,\ldots,t_n,t_1)+\ldots+\dfrac{1}{n}(t_n,t_{n-1},\ldots,t_1)\right)\\
&\le\dfrac{1}{n}\psi(t_1,t_2,\ldots,t_n)+\dfrac{1}{n}\psi(t_2,\ldots,t_n,t_1)+\ldots+\dfrac{1}{n}\psi(t_n,t_{n-1},\ldots,t_1)\\
&=\psi(t_1,\ldots,t_n).
\end{align*}	
Thus, $(\frac{1}{n},\ldots,\frac{1}{n})$ is a minimizer of $\psi$.
The proof is complete.
\end{proof}	

The next statement provides primal and dual norm constructions on product spaces via elements of $\pmb{\Psi}_n$; see \cite{Cuo25}.
\begin{theorem}\label{T2.3}
Let $(X,\|\cdot\|)$ be a normed space, and $\psi\in\pmb{\Psi}_n$.
Define
\begin{equation}\label{T2.2-1}
\vertiii{x}_\psi
:= \begin{cases}
\left(\sum_{i=1}^n\|x_i\|\right)\cdot\psi\left(\dfrac{\|x_1\|}{\sum_{i=1}^n\|x_i\|},\ldots,\dfrac{\|x_{n}\|}{\sum_{i=1}^n\|x_i\|}\right)  & \text{if } x\ne0_{X^n},\\
0 & \text{otherwise}
\end{cases} 
\end{equation}
for all $x:=(x_1,\ldots,x_n)\in X^n$.
The following assertions hold:
\begin{enumerate}
\item\label{T2.3-1}
$\vertiii{\cdot}_\psi$ is a norm on $X^n$.
Moreover, 
$\vertiii{\cdot}_\psi$ is strictly convex if and only if $\|\cdot\|$ and $\psi$ are strictly convex;
\item\label{T2.3-3}
the corresponding dual norm is given by
\begin{gather}\label{T2.4-1}
\vertiii{x^*}^*_{\psi} =\max_{(t_1,\ldots,t_n)\in\Omega_n}\dfrac{\sum_{i=1}^nt_i\|x^*_i\|^*}{\psi (t_1,\ldots,t_{n})}
\end{gather}	
for all $x^*:=(x^*_1,\ldots,x^*_n)\in(X^*)^n$.
\end{enumerate}	
\end{theorem}

\begin{remark}\label{R2.5}
\begin{enumerate}
\item\label{R2.5-2}
Let $\psi\in\pmb{\Psi}_n$.
Define a conjugate-type function of $\psi$  by
\begin{gather}\label{psi1}
	\psi^*(s):=\max_{t\in\Omega_n}\dfrac{\langle t,s\rangle}{\psi(t)}
\end{gather}	
for all $s\in\Omega_n$.
By Theorem~\ref{T2.3}\eqref{T2.3-3}, we have  $\vertiii{\cdot}^*_\psi=\vertiii{\cdot}_{\psi^*}$.
 Moreover, one can verify that $\psi^*\in\pmb{\Psi}_n$; see \cite{Cuo25,Dho05,Mit05}.
\item
Let $\vertiii{\cdot}_p$ and $\psi_p$ be given by \eqref{pnorm} and \eqref{ppsi}, respectively. 
By \eqref{T2.2-1}, we have $\vertiii{\cdot}_p=\vertiii{\cdot}_{\psi_p}$.
In view of \eqref{T2.4-1} and \eqref{psi1},
\begin{equation*}
\psi^*_{p}(s_1,\ldots,s_n)
= \begin{cases}
\max\{s_1,\ldots,s_n\}	& \text{if } p=1,\\
\left(s_1^q+\ldots+s_n^q\right)^{\frac{1}{q}}  & \text{if } p\in(1,\infty]
\end{cases} 
\end{equation*}
for all $(s_1,\ldots,s_n)\in\Omega_n$.
\item\label{R2.5-1}
By Proposition~\ref{P2.2}\eqref{P2.2-1} and definition 	\eqref{T2.2-1}, we have $\vertiii{\cdot}_\infty\le \vertiii{\cdot}_\psi\le \vertiii{\cdot}_1$.
\end{enumerate}
\end{remark}	

The following statements play a key role of our study.
\begin{proposition}\label{P3.5+}
Let $(X,\|\cdot\|)$ be a normed space, $\psi\in\pmb{\Psi}_n$, $x:=(x_1,\ldots,x_n)\in X^n$ and $x^*:=(x^*_1,\ldots,x^*_n)\in (X^*)^n$.
The following assertions hold.
\begin{enumerate}
\item\label{P3.5-1}
$\sum_{i=1}^{n}\|x^*_i\|^*\cdot\|x_i\|\le\vertiii{x^*}_{\psi^*}\cdot\vertiii{x}_\psi.$
\item\label{P3.5-2}
If 
\begin{gather}\label{P2.6-2.1}
\ang{x^*,x}=\vertiii{x^*}_{\psi^*}\cdot\vertiii{x}_\psi,
\end{gather}	
then 
\begin{gather}\label{P2.6-2.2}
\ang{x^*_i,x_i}= \|x^*_i\|^*\cdot\|x_i\|\;\;(i=1,\ldots,n).
\end{gather}	
\item\label{P3.6-1}
Equality \eqref{P2.6-2.1} holds for $\psi:=\psi_\infty$ if and only if condition \eqref{P2.6-2.2} is satisfied, and 
\begin{gather}\label{P2.6-4.1}
\left(\|x_i\|-\max_{1\le i\le n}\|x_i\|\right)\cdot\|x^*_i\|^*=0\;\;(i=1,\ldots,n).
\end{gather}
\item\label{P3.6-2}
Equality \eqref{P2.6-2.1} holds for $\psi:=\psi_1$ if and only if condition \eqref{P2.6-2.2} is satisfied, and 
\begin{gather}\label{P2.6-3.1}
\left(\|x^*_i\|^*-\max_{1\le i\le n}\|x^*_i\|^*\right)\cdot\|x_i\|=0\;\;(i=1,\ldots,n).
\end{gather}
\item\label{P3.6-3}
Equality \eqref{P2.6-2.1} holds for $\psi:=\psi_p$ with $p\in(1,\infty)$  if and only if condition \eqref{P2.6-2.2} is satisfied, and 
\begin{gather}\label{P2.6-5.1}
\dfrac{\|x^*_i\|^{*q}}{\sum_{i=1}^{n}\|x^*_i\|^{*q}} =\dfrac{\|x_i\|^p}{\sum_{i=1}^{n}\|x_i\|^p}\;\;(i=1,\ldots,n).
\end{gather}
\end{enumerate}
\end{proposition}

\begin{proof}
\begin{enumerate}	
\item 
By \eqref{psi1}, $\ang{t,s}\le\psi(t)\psi^*(s)$ for all $t,s\in\Omega_n$.
Applying this inequality for $t:=\left(\frac{\|x_1\|}{\sum_{i=1}^{n}\|x_i\|},\ldots,\frac{\|x_n\|}{\sum_{i=1}^{n}\|x_i\|}\right)$ and $s:=\left(\frac{\|x^*_1\|^*}{\sum_{i=1}^{n}\|x^*_i\|^*},\ldots,\frac{\|x^*_n\|^*}{\sum_{i=1}^{n}\|x^*_i\|^*}\right)$, we obtain
\begin{gather*}
	\sum_{i=1}^{n}\|x^*_i\|^*\cdot\|x_i\|\le \left(\sum_{i=1}^{n}\|x^*_i\|^*\right)\psi^*(s)\cdot \left(\sum_{i=1}^{n}{\|x^*_i\|^*}\right) \psi(t)\overset{ \eqref{T2.2-1}}{=} \vertiii{x^*}_{\psi^*}\cdot\vertiii{x}_\psi.
\end{gather*}
\item 
is a direct consequence of  assertion  \eqref{P3.5-1} and the fact  $\ang{x^*,x}\le\sum_{i=1}^{n}{\|x^*_i\|^*}\cdot\|x_i\|$.
\item
In view of assertion \eqref{P3.5-1} with $\psi:=\psi_\infty$, we have
\begin{gather}\label{P2.6-7}
	\sum_{i=1}^{n}\|x^*_i\|^*\cdot\|x_i\|\le \max_{1\le i\le n}\|x_i\|\cdot \sum_{i=1}^{n}\|x^*_i\|^*.
\end{gather}	
Suppose that
\begin{gather}\label{P2.6-11}
	\ang{x^*,x}=\max_{1\le i\le n}\|x_i\|\cdot \sum_{i=1}^{n}\|x^*_i\|^*.
\end{gather}	
Then condition \eqref{P2.6-2.2} is satisfied.
If $x^*_i\ne 0$ and $\|x_i\|<\max_{1\le j\le n}\|x_j\|$ for some $i\in\{1,\ldots,n\}$, then 
inequality of \eqref{P2.6-7} is strict, which contradicts \eqref{P2.6-11}.
Thus, condition \eqref{P2.6-4.1} is satisfied.
Conversely, if conditions \eqref{P2.6-2.2} and \eqref{P2.6-4.1} are satisfied, then it is clear that condition \eqref{P2.6-11} holds true.
\item 
In view of assertion \eqref{P3.5-1} with $\psi:=\psi_1$, we have
\begin{gather}\label{P2.6-6}
	\sum_{i=1}^{n}\|x^*_i\|^*\cdot\|x_i\|\le \max_{1\le i\le n}\|x^*_i\|^*\cdot \sum_{i=1}^{n}\|x_i\|.
\end{gather}	
Suppose that 
\begin{gather}\label{P2.6-10}
	\ang{x^*,x}=\max_{1\le i\le n}\|x^*_i\|^*\cdot \sum_{i=1}^{n}\|x_i\|.
\end{gather}	
Then condition \eqref{P2.6-2.2} is satisfied.
If $x_i\ne 0$ and $\|x^*_i\|^*<\max_{1\le j\le n}\|x^*_j\|^*$ for some $i\in\{1,\ldots,n\}$, then 
inequality of \eqref{P2.6-6} is strict, which contradicts \eqref{P2.6-10}.
Thus, condition \eqref{P2.6-3.1} is satisfied.
Conversely, if conditions \eqref{P2.6-2.2} and \eqref{P2.6-3.1} are satisfied, then it is clear that condition \eqref{P2.6-10} holds true.
\item 
In view of assertion \eqref{P3.5-1} with $\psi:=\psi_p$, we have
\begin{gather*}
	\sum_{i=1}^{n}\|x^*_i\|^*\cdot\|x_i\|\le \left(\sum_{i=1}^{n}\|x^*_i\|^{*q}\right)^{\frac{1}{q}}\cdot \left(\sum_{i=1}^{n}\|x_i\|^{p}\right)^{\frac{1}{p}}.
\end{gather*}	
Suppose that
\begin{gather}\label{P2.6-9} \ang{x^*,x}=\left(\sum_{i=1}^{n}\|x^*_i\|^{*q}\right)^{\frac{1}{q}}\cdot \left(\sum_{i=1}^{n}\|x_i\|^{p}\right)^{\frac{1}{p}}.
\end{gather}
Then condition \eqref{P2.6-2.2} is satisfied, and 
$\|x_i\|^p=\lambda\|x^*_i\|^{*q}$ $(i=1,\ldots,n)$ for some $\lambda\ge 0$.
We have
\begin{gather*}
	\sum_{i=1}^{n}\|x_i\|^p=\lambda\cdot\sum_{i=1}^{n}\|x^*_i\|^{*q}\Rightarrow \lambda= \frac{\sum_{i=1}^{n}\|x_i\|^p}{\sum_{i=1}^{n}\|x^*_i\|^{*q}}.
\end{gather*}	
Thus, condition \eqref{P2.6-5.1} is satisfied. 
Conversely, suppose that conditions \eqref{P2.6-2.2} and \eqref{P2.6-5.1} are satisfied.
Then
\begin{align*}
	\ang{x^*,x}
	=\sum_{i=1}^{n}\|x^*_i\|^*\cdot\|x_i\|
	=\left(\dfrac{\sum_{i=1}^{n}\|x^*_i\|^{*q}}{\sum_{i=1}^{n}\|x_i\|^{p}}\right)^\frac{1}{q}\sum_{i=1}^{n}\|x_i\|^{p}
	=\left(\sum_{i=1}^{n}\|x^*_i\|^{*q}\right)^{\frac{1}{q}}\cdot \left(\sum_{i=1}^{n}\|x_i\|^{p}\right)^{\frac{1}{p}}.
\end{align*}	
Thus, condition \eqref{P2.6-9} is satisfied.
\end{enumerate}
The proof is complete.
\end{proof}

\begin{lemma}\label{L1.1}
Let $(X,\|\cdot\|)$ be a normed  space, and $x^*\in X^*$.
Define
\begin{gather}\label{A+}
{\text{\rm T}}_{\|\cdot\|}(x^*):=\{x\in X\mid \langle x^*, x\rangle=\|x^*\|^*\cdot \|x\|\}.
\end{gather}
The following assertions hold.
\begin{enumerate}
\item\label{L1.1-1}
If $X$ is a Hilbert space, then there exists a unique vector $\bx\in X$ such that
\begin{gather*}
{\text{\rm T}}_{\|\cdot\|}(x^*)=\left\{\lambda \bx\mid \lambda\ge 0\right\}.
\end{gather*}	
\end{enumerate}
Suppose that $X:=\R^n$ and $x^*:=(x^*_{1},\ldots,x^*_{n})$.
\begin{enumerate}[start=2]
\item\label{L1.1-4}
If  $\|\cdot\|:=\|\cdot\|_\infty$, then 
\begin{gather*}
{\text{\rm T}}_{\|\cdot\|_\infty}(x^*)=\left\{(x_1,\ldots,x_n)\in\R^n\mid 
x^*_{i}\cdot x_{i}\ge 0,\; |x_{i}|=\|x\|_\infty\;\;\text{if}\;\;x^*_{i}\ne 0
\right\}.
\end{gather*}	
\item\label{L1.1-3}
If $\|\cdot\|:=\|\cdot\|_1$, then 
\begin{gather*}
{\text{\rm T}}_{\|\cdot\|_1}(x^*)=\left\{(x_1,\ldots,x_n)\in\R^n\mid
x^*_{i}\cdot x_{i}\ge 0,\; x_i=0\;\;\text{if}\;\; |x^*_i|<\|x^*\|_\infty
\right\}.
\end{gather*}	
\item\label{L1.1-2}
If $\|\cdot\|:=\|\cdot\|_p$ with $p\in(1,\infty)$, then 
\begin{gather*}
{\text{\rm T}}_{\|\cdot\|_p}(x^*)=\left\{\left(\lambda\text{\rm sgn}(x^*_1) |x^{*}_{1}|^{\frac{q}{p}},\ldots,\lambda \text{\rm sgn}(x^*_n)  |x^{*}_{n}|^{\frac{q}{p}}\right)\in\R^n\mid \lambda\ge 0\right\}.
\end{gather*}	
\end{enumerate}
\end{lemma}	

\begin{proof}
\begin{enumerate}
\item 
By the Riesz representation theorem \cite[Theorem~5.5]{Bre11}, there exists a unique 
vector $\bar x\in X$ such that 
\begin{gather*}
\|x^*\|^*=\|\bar x\|\;\;\text{and}\;\;
\langle x^*,x\rangle =\langle \bar x,x\rangle\;\;\text{for all}\;\;x\in X.
\end{gather*}	
Let $x\in {\text{\rm T}}_{\|\cdot\|}(x^*)$.
Then $\langle \bx,x\rangle=\|\bar x\|\cdot\|x\|$.
Thus, there exists a scalar $\lambda\ge 0$ such that $x=\lambda \bx$.
Conversely, let $x:=\lambda \bx$ for some $\lambda\ge 0$.
Then
\begin{gather*}
\langle x^*, x\rangle=\langle \bx,x\rangle=\langle \bx, \lambda \bx\rangle=\|\bx\|\cdot\lambda\|\bx\|=\|x^*\|^*\cdot\|x\|.
\end{gather*}	
Thus $x\in {\text{\rm T}}_{\|\cdot\|}(x^*)$.
\item
is a direct consequence of Proposition~\ref{P3.5+}\eqref{P3.6-1}.
\item 
is a direct consequence of Proposition~\ref{P3.5+}\eqref{P3.6-2}.
\item
Let $x:=(x_1,\ldots,x_n)\in {\text{\rm T}}_{\|\cdot\|_p}(x^*)$.
Then
\begin{gather}\label{T2.12-4}
\ang{x^*,x}=\left(\sum_{i=1}^{n}|x_{i}|^p\right)^{\frac{1}{p}}\cdot\left(\sum_{i=1}^{n}|x^*_{i}|^q\right)^{\frac{1}{q}}.
\end{gather}	
By Proposition~\ref{P3.5+}\eqref{P3.6-3}, we have
$x^*_{i}\cdot x_i\ge 0$ and  $|x_{i}|^p=\gamma\cdot|x^*_{i}|^q$ $(i=1,\ldots,n)$ for some $\gamma\ge 0$. 
Thus $x_i=\lambda \text{sgn}(x^*_i)|x^{*}_{i}|^{\frac{q}{p}}$ $(i=1,\ldots,n)$ with $\lambda:=\gamma^{\frac{1}{p}}$.
Conversely, let $x:=(x_1,\ldots,x_n)$ with 
$x_i:=\lambda\text{sgn}(x^*_i)|x^{*}_{i}|^{\frac{q}{p}}$ $(i=1,\ldots,n)$ for some $\lambda\ge 0$.
Then $x^*_{i}\cdot x_i\ge 0$ {$(i=1,\ldots,n)$} and
\begin{align*}
\ang{x^*,x}
&=\sum_{i=1}^{n}|x^*_{i}|\cdot|x_{i}|
=\lambda\cdot\sum_{i=1}^{n}|x^*_{i}|^q\\
&=\lambda \left(\sum_{i=1}^{n}|x^*_{i}|^q\right)^{\frac{1}{p}}\cdot
\left(\sum_{i=1}^{n}|x^*_{i}|^q\right)^{\frac{1}{q}}
=\left(\sum_{i=1}^{n}|x_{i}|^p\right)^{\frac{1}{p}}\cdot \left(\sum_{i=1}^{n}|x^*_{i}|^q\right)^{\frac{1}{q}}. 
\end{align*}	
Thus, condition \eqref{T2.12-4} is satisfied.
Hence, $x\in {\text{\rm T}}_{\|\cdot\|_p}(x^*)$.
\end{enumerate}
The proof is complete.
\end{proof}	

\begin{remark}
Lemma~\ref{L1.1} provides explicit formulas for the set~\eqref{A+} in some particular cases in finite and infinite dimensional spaces. 
Although the proof uses standard arguments from functional analysis, 
to the best of our knowledge, such descriptions  have not been explicitly presented in the literature.
\end{remark}

\section{General Norm Minimization Problem}\label{S3}
Let $X$ be a vector space, and $v_1,\ldots,v_n\in X$ be distinct vectors.
Consider the problem:
\begin{gather}
\label{P}
\tag{$P$}
{\rm{minimize }}\;\;f_{\vertiii{\cdot}}(u):=\vertiii{(u-v_1,\ldots,u-v_n)}\;\;\text{for all}\;\;
u\in X,
\end{gather}
where $\vertiii{\cdot}$ is a given norm on $X^n$.
Denote $\text{\rm Sol}(P,\vertiii{\cdot})$ the set of all solutions of problem \eqref{P}.
When the norm on $X^n$ is known from the context, we shall simply write $\text{\rm {Sol}}(P)$.

\begin{proposition}[Existence and uniqueness]\label{P3.2}
Suppose that $(X,\|\cdot\|)$ is a reflexive Banach space,
$f_{\vertiii{\cdot}}(u)\to\infty$ as $\|u\|\to\infty$, and
\begin{gather}\label{P3.2-10}
 \vertiii{(u,\ldots,u)}\le \gamma\|u\|\;\;\text{for all}\;\;u\in X\;\;\text{and some } \gamma>0.
 \end{gather}
Then $\text{\rm {Sol}}(P,\vertiii{\cdot})$ is a nonempty, closed and convex set. 
Moreover, if  $f_{\vertiii{\cdot}}$ is strictly convex, then the solution set is a singleton.
\end{proposition}	

\begin{proof}
Observe that $f_{\vertiii{\cdot}}$ is a composition of the  {affine} mapping $L: X\to X^n$ given by
\begin{gather}\label{A}
L(u):=(u-v_1,\ldots,u-v_n)\;\;\text{for all}\;\;u\in X,
\end{gather}	
and the norm function
\begin{gather}\label{phi}
\phi(x):=\vertiii{x} \;\;\text{for all}\;\;x\in X^n.
\end{gather}	
By \eqref{P3.2-10},
\begin{gather*}
\vertiii{L(u)-L(w)}=\vertiii{(u-w,\ldots,u-w)}\le \gamma\|u-w\|\;\;\text{for all}\;\;u,w\in X.
\end{gather*}	
Thus, $L$ is globally Lipschitz.
Hence, $f_{\vertiii{\cdot}}$ is convex and continuous.
By \cite[Corollary~3.23]{Bre11}, $\text{\rm {Sol}}(P,\vertiii{\cdot})$ is nonempty. 
The continuity and convexity  of $f_{\vertiii{\cdot}}$ imply that $\text{\rm {Sol}}(P,\vertiii{\cdot})$ is closed and convex.
The uniqueness under the strict convexity is a standard fact of convex analysis.
\end{proof}	

\begin{remark}\label{P4.2}
\begin{enumerate}
\item		
Problem~\eqref{P} serves as a unified framework encompassing several classical optimization problems.
In Sections~\ref{S4}–\ref{S6}, we study three important cases of problem~\eqref{P}, where  $(X,\|\cdot\|)$ is a normed space, and $\vertiii{\cdot}$ is  the sum, max and $p$-norms.
\item
A closely related model is studied in \cite{Ves97}.
The author considers
 the objective function of the form:
\begin{gather*}
f(u):=\varphi(\|u-v_1\|,\ldots,\|u-v_n\|)\;\;
\text{for all}\;\;u\in X,
\end{gather*}	 
where $(X,\|\cdot\|)$ is a Banach space and  $\varphi: \R^n_+\to\R_+$ is a continuous, monotone, convex and coercive function.
This study, among other things, is based on the method of finitely intersecting balls \cite[Theorem~2.7]{Ves97} which is  further examined in \cite{BanRao00}. 
In particular, the author proves the nonemptiness of the solution set under the assumption that $X$ is norm-one complemented in its bidual \(X^{**}\) \cite[Proposition~2.2]{Ves97}. 
It is worth emphasizing that the class of Banach spaces that are norm-one complemented in their bidual strictly contains the class of reflexive spaces \cite[Remark~2.3]{Ves97}. 
Further results concerning the uniqueness and compactness of the solution set for this model can be found in \cite[Theorems~3.2 \&~3.4]{Ves97}.
\end{enumerate}
\end{remark}	

The next statement gives some sufficient conditions for the assumptions in Proposition~\ref{P3.2}.
\begin{proposition}\label{P4.3}
Suppose that $(X,\|\cdot\|)$ is a normed space, and $\psi\in \pmb{\Psi}_n$. 
The following assertions hold true.
\begin{enumerate}
\item\label{P4.3-1}
$f_{\vertiii{\cdot}_\psi}(u)\to\infty$ as $\|u\|\to\infty$.
\item\label{P4.3-10} $\vertiii{(u,\ldots,u)}_\psi=n\psi\left(\frac{1}{n},\ldots,\frac{1}{n}\right)\|u\|$ for all $u\in X$.
\item 
If $\psi$ and $\|\cdot\|$ are strictly convex, then $f_{\vertiii{\cdot}_{\psi}}$ is strictly convex.
\item\label{P4.3-3}
If $\|\cdot\|$ is strictly convex, $v_1,\ldots,v_n$ are not collinear, and $\psi:=\psi_1$ is given by \eqref{E3.1-1}, then  $f_{\vertiii{\cdot}_{\psi}}$ is strictly convex.
\item \label{P4.3-4}
If \rm{dim}\;$X<\infty$, then $\text{\rm {Sol}}(P,\vertiii{\cdot}_\psi)$  is compact.
\end{enumerate}	
\end{proposition}	

\begin{proof}
\begin{enumerate}
\item 
By Remark~\ref{R2.5}\eqref{R2.5-1},
\begin{gather*}
f_{\vertiii{\cdot}_\psi}(u)=\vertiii{(u-v_1,\ldots,u-v_n)}_\psi \ge \max_{1\le i\le n}\|u-v_i\|\ge \|u\|- \max_{1\le i\le n}\|v_i\|
\end{gather*}	
for all $u\in X$.
Thus $f_{\vertiii{\cdot}_\psi}(u)\to\infty$ as $\|u\|\to\infty$.
\item 
is a direct consequence of \eqref{T2.2-1}.
\item 
Suppose that $\psi$ and $\|\cdot\|$ are strictly convex.
By Theorem~\ref{T2.3}\eqref{T2.3-1}, the norm $\vertiii{\cdot}_\psi$ is strictly convex.
Then $f_{\vertiii{\cdot}_{\psi}}$ is also strictly convex since it is a composition of the norm  $\vertiii{\cdot}_\psi$ and the affine map  \eqref{A}.
\item 
We have $f_{\vertiii{\cdot}_{\psi}}(u)=\sum_{i=1}^{n}\|u-v_i\|$ for all $u\in X$.
Suppose by contradiction that there exist  $\bar u,\bar v\in X$ with $\bar u\ne \bar v$ and a scalar $\lambda\in(0,1)$ such that
\begin{gather*}
f_{\vertiii{\cdot}_{\psi}}(\lambda\bar u+(1-\lambda)\bar v)=\lambda f_{\vertiii{\cdot}_{\psi}}(\bar u)+(1-\lambda)f_{\vertiii{\cdot}_{\psi}}(\bar v).
\end{gather*}
Thus 
\begin{gather}\label{P4.3-6}
\|\lambda (\bar u-v_i)+(1-\lambda)(\bar v-v_i)\|=\|\lambda (\bar u-v_i)\|+\|(1-\lambda)(\bar v-v_i)\|
\end{gather}	
for all  $i=1,\ldots,n$.
Let  $L(\bar u,\bar v):=\{t\bar u+(1-t)\bar v\mid t\in\R\}$ be the  affine line passing through $\bar u$ and $\bar v$.
Fix $i\in\{1,\ldots,n\}$. 
If either $v_i=\bar u$ or $v_i=\bar v$, then $v_i\in L(\bar u,\bar v)$.
Suppose that $v_i\ne \bar u$ and $v_i\ne \bar v$.
By \eqref{P4.3-6}, there exists a scalar $\xi>0$ such that
$\lambda (\bar u-v_i)=\xi (1-\lambda)(\bar v-v_i)$. 
Then $\bar u-v_i=\gamma_i (\bar v-v_i)$ with $\gamma_i:=\frac{\xi (1-\lambda)}{\lambda}>0$.
Note that $\gamma_i\ne 1$ since $\bar u\ne\bar v$.
We have
\begin{gather*}
v_i=\frac{1}{1-\gamma_i}\cdot\bar u-\frac{\gamma_i}{1-\gamma_i}\cdot\bar v\in L(\bar u,\bar v).
\end{gather*}	
Thus, $v_i\in  L(\bar u,\bar v)$ for all $i=1,\ldots,n$, which contradicts the non-collinearity assumption.
\item
Suppose that \rm{dim}\;$X<\infty$.
By Proposition~\ref{P3.2}, the solution set is closed.
We now prove that it is bounded.
By Remark~\ref{R2.5}\eqref{R2.5-1},
\begin{align}\label{R2.5-10}
f_{\vertiii{\cdot}_{\psi}}(u)
\ge \max_{1\le i\le n}\|u-v_i\|
\ge \dfrac{1}{n}\sum_{i=1}^{n}\|u-v_i\|
\ge \|u\|-\dfrac{1}{n}\sum_{i=1}^{n}\|v_i\|
\end{align}	
for all $u\in X$.
If $\|u\|> \left(1+\frac{1}{n}\right)\sum_{i=1}^{n}\|v_i\|$, then by \eqref{R2.5-10} and Remark~\ref{R2.5}\eqref{R2.5-1},
\begin{gather*}
f_{\vertiii{\cdot}_{\psi}}(u)>\sum_{i=1}^{n}\|v_i\|\ge \vertiii{(v_1,\ldots,v_n)}_\psi=f_{\vertiii{\cdot}_{\psi}}(0_X),
\end{gather*}
and consequently, $u$ is not a minimizer of $f_{\vertiii{\cdot}_{\psi}}$.
Hence, if $u\in \text{\rm {Sol}}(P,\vertiii{\cdot}_\psi)$, then $\|u\|\le \left(1+\frac{1}{n}\right)\sum_{i=1}^{n}\|v_i\|$.
\end{enumerate}		
The proof is complete.
\end{proof}	

\begin{remark}
\begin{enumerate}
\item		
When $\psi:=\psi_1$ is given by \eqref{ppsi}, Proposition~\ref{P4.3}\eqref{P4.3-4} recaptures \cite[Proposition~3.2]{MarSwaWei02}.
\item 
If $X$ is a finite dimensional space equipped with the Euclidean norm, then assertion \eqref{P4.3-3} of Proposition~\ref{P4.3} recaptures \cite[Proposition~4.30]{MorNam14}.
\end{enumerate}
\end{remark}	

The following result provides a formula for the subdifferential of the objective function of problem~\eqref{P}.
\begin{proposition}\label{P2.4}
Let $X$ be a normed space, $\vertiii{\cdot}$ be a norm on $X^n$, and $\bar u\in X$.
Suppose that condition \eqref{P3.2-10} is satisfied.
Then
\begin{align*}
\partial f_{\vertiii{\cdot}}(\bar u)=\Big\{\sum_{i=1}^{n}x^*_i\mid \vertiii{(x^*_1,\ldots,x^*_n)}^*=1,\;
\sum_{i=1}^{n}\langle x^*_i,\bar u -v_i\rangle=\vertiii{(\bar u-v_1,\ldots,\bar u-v_n)}\Big\}.
\end{align*}
\end{proposition}	

\begin{proof}
Recall that the convex function $f_{\vertiii{\cdot}}$ is a composition of the affine mapping $L$ and the norm function $\phi$ given by \eqref{A} and \eqref{phi}, respectively.
Let $v:=(v_1,\ldots,v_n)$.
Define 
\begin{gather*}
g(x):=\phi(x-v)\;\;\text{for all}\;\;x\in X^n\;\;\text{and}\;\;
A(u):=(u,\ldots,u)\;\;\text{for all}\;\;u\in X.
\end{gather*}	
Then
\begin{gather}\label{P2.4-10}
f_{\vertiii{\cdot}}(u)=\vertiii{A(u)-v}=(g\circ A)(u)\;\;\text{for all}\;\;u\in X.	
\end{gather}	
The function $g$ is convex and continuous.
By \eqref{P3.2-10}, the linear operator $A$ is continuous.
The adjoint operator $A^*:(X^*)^n\to X^*$ is given by
\begin{gather}\label{P2.4-12}
A^*(x^*_1,\ldots,x^*_n)= \sum_{i=1}^{n}x^*_i\;\;
\text{for all}\;\;(x^*_1,\ldots,x^*_n)\in(X^*)^n.
\end{gather}	
By \eqref{P2.4-10} and \cite[Theorem~3.40]{Pen13}, 
\begin{gather}\label{P2.4-11}
\partial f_{\vertiii{\cdot}}(\bar u)=\partial (g\circ A)(\bar u)=A^*\partial g(A(\bar u)).
\end{gather}	
Since $g(\cdot)=\phi(\cdot-v)$, it follows from the definition of the convex subdifferential that $\partial g(x)=\partial \phi(x-v)$ for all $x\in X^n$.
In particular, 
\begin{gather}\label{P2.4-13}
\partial {g(A(\bar u))}=\partial \phi (A(\bar u)-v)=\partial\phi(L(\bar u)).
\end{gather}
Since $v_1,\ldots,v_n$ are distinct, we have  $L(\bar u)\ne 0_{X^n}$.
By \eqref{sn},
\begin{gather*}
\partial \phi(L(\bar u))=\left\{x^*:=(x^*_1,\ldots,x^*_n)\mid \vertiii{x^*}^*=1,\;\sum_{i=1}^{n}\langle x^*_i,\bar u -v_i\rangle=\vertiii{(\bar u-v_1,\ldots,\bar u-v_n)}\right\}.
\end{gather*}	
By \eqref{P2.4-12}, \eqref{P2.4-11} and \eqref{P2.4-13}, we obtain the desired equality.
The proof is complete.
\end{proof}	

The next result establishes dual necessary and sufficient optimality conditions for problem~\eqref{P}.
\begin{theorem}\label{T3.5}
Let $X$ be a normed space,  and $\vertiii{\cdot}$ be a norm on $X^n$.
Suppose that condition \eqref{P3.2-10} is satisfied.
A vector $\bar u\in \text{\rm Sol}(P,\vertiii{\cdot})$ if and only if there exist $x^*_1,\ldots,x^*_n\in X^*$ such that
\begin{gather}\label{P2.5-0}
\displaystyle \sum_{i=1}^n x_i^*=0,\;\;
\vertiii{(x^*_1,\ldots,x^*_n)}^*=1,\;\;
\displaystyle \sum_{i=1}^n \langle x_i^*,\bar u-v_i\rangle=\vertiii{(\bar u-v_1,\ldots,\bar u-v_n)}.
\end{gather}	
Moreover,
\begin{gather}\label{P2.5-1}
\text{\rm Sol}(P,\vertiii{\cdot})=\left\{u\in X\mid \sum_{i=1}^{n}\ang{x^*_i,u-v_i}=\vertiii{(u-v_1,\ldots,u-v_n)}\right\}.
\end{gather}
\end{theorem}	

\begin{proof}
We have $\bar u\in \text{\rm Sol}(P,\vertiii{\cdot})$ if and only if $0\in \partial f_{\vertiii{\cdot}}(\bar u)$.
{Then the} characterization \eqref{P2.5-0} is a consequence of 
Proposition~\ref{P2.4}.

Let $\text{\rm R}$ denote the set in the right-hand side of \eqref{P2.5-1}.
Suppose that $u\in\text{\rm R}$.
Then
\begin{gather*}
\sum_{i=1}^{n}\ang{x^*_i,u-v_i}=\vertiii{(u-v_1,\ldots,u-v_n)}.	
\end{gather*}	
By Proposition~\ref{P2.4}, we have $0\in\partial f_{\vertiii{\cdot}}(u)$.
Thus, $u\in \text{\rm Sol}(P,\vertiii{\cdot})$, and consequently,
 $\text{\rm R}\subset\text{\rm Sol}(P,\vertiii{\cdot})$.
Suppose that $u\in\text{\rm Sol}(P,\vertiii{\cdot})$. Then
\begin{align*}
\sum_{i=1}^n\ang{x^*_i,u-v_i}	
&=\sum_{i=1}^n\ang{x^*_i,u-\bar u}+
\sum_{i=1}^n\ang{x^*_i,\bar u-v_i}\\
&=\ang{\sum_{i=1}^{n}x^*_i,u-\bar u}+\vertiii{(\bar u -v_1,\ldots,\bar u-v_n)}
=\vertiii{(u-v_1,\ldots,u-v_n)},
\end{align*}	
where the last equality follows from the fact that $\sum_{i=1}^n x_i^*=0$ and $f_{\vertiii{\cdot}}(\bar u)=f_{\vertiii{\cdot}}(u)$.
Thus $u\in \text{\rm R}$, and consequently, $\text{\rm Sol}(P,\vertiii{\cdot})\subset \text{\rm R}$.
\end{proof}	

\begin{remark}
\begin{enumerate}
\item		
In view of  \eqref{P2.5-1}, a complete description of the solution set can be obtained if we know a solution $\bar u$
and its corresponding dual vectors $x^*_1,\ldots,x^*_n$ given by \eqref{P2.5-0}.
\item 
Let $X$ be a normed space, and $\psi\in \pmb{\Psi}_n$. 
By Proposition~\ref{P4.3}\eqref{P4.3-10}, the norm $\vertiii{\cdot}_\psi$  satisfies condition \eqref{P3.2-10}.
\end{enumerate}
\end{remark}

The following example illustrates  Theorem~\ref{T3.5} for the case $n=2$.
\begin{example}\label{E3.7}
Let $(X,\|\cdot\|)$ be a normed space, $v_1,v_2\in X$ with $v_1\ne v_2$, and $\bar u:=\frac{v_1+v_2}{2}$.
Suppose that $\psi\in \pmb{\Psi}_2$ is permutation-symmetric, $\vertiii{\cdot}_\psi$ and $\psi^*$ are given by \eqref{T2.2-1} and \eqref{psi1}, respectively.
Then there exist $x^*_1,x^*_2\in X^*$ such that
\begin{gather}\label{E3.7-4}
x_1^*+x^*_2=0,\;\;
\vertiii{(x^*_1,x^*_2)}_{\psi^*}=1,\;\;
\ang{x_1^*,\bar u-v_1}+\ang{x_2^*,\bar u-v_2}=\vertiii{(\bar u-v_1,\bar u-v_2)}_\psi.
\end{gather}	
By Theorem~\ref{T3.5}, $\bar u\in \text{\rm Sol}(P,\vertiii{\cdot}_\psi)$ and
\begin{gather*}
\text{\rm Sol}(P,\vertiii{\cdot}_\psi)=\left\{u\in X\mid \ang{x^*_1,u-v_1}+\ang{x^*_2,u-v_2}=\vertiii{(u-v_1,u-v_2)}_\psi\right\}.
\end{gather*}
The optimal value is
$f_{\min}=\frac{1}{2}\cdot\vertiii{(v_1-v_2,v_2-v_1)}_\psi
=\|v_1-v_2\|\cdot\psi(1/2,1/2).$
\end{example}	

\begin{proof}
The objective function is given by
\begin{gather*}
f_{\vertiii{\cdot}_{\psi}}(u):=\vertiii{(u-v_1,u-v_2)}_\psi\;\;\text{for all}\;\;u\in X.
\end{gather*}	
We are going to show that $0\in\partial f_{\vertiii{\cdot}_{\psi}}(\bar u)$.
Let $x_i:=\bar u-v_i$ $(i=1,2)$ and $\mathrm{m}:=\psi(1/2,1/2)>0$.
Then $x_1+x_2=0$.
By \eqref{T2.2-1}, we have
$\|(x_1,x_2)\|_\psi=\|(x_1,-x_1)\|_\psi=2\mathrm{m}\|x_1\|.$
By the Hahn-Banach Theorem \cite[Corollary~1.3]{Bre11}, there exists a vector $x^*\in X^*$ such that $\|x^*\|^*=1$ and $\langle x^*,x_1\rangle=\|x_1\|$.
Let $x^*_1:=\mathrm{m} x^*$ and $x^*_2:=-\mathrm{m} x^*$.
Then $x^*_1+x^*_2=0$, and
\begin{gather*}
\langle(x^*_1,x^*_2),(x_1,x_2)\rangle
=\mathrm{m}\ang{x^*,x_1}+\mathrm{m}\ang{-x^*,-x_1}
=2\mathrm{m}\langle x^*,x_1\rangle=2\mathrm{m}\|x_1\|=\|(x_1,x_2)\|_\psi.
\end{gather*}	
By Proposition~\ref{P2.2}\eqref{P2.2-2}, we have $\min\limits_{(t_1,t_2)\in\Omega_2}\psi(t_1,t_2)=\mathrm{m}$.
By Theorem~\ref{T2.3}\eqref{T2.3-3} and Remark~\ref{R2.5}\eqref{R2.5-2},
\begin{align*}
\vertiii{(x^*_1,x^*_2)}^*_{\psi}=
\vertiii{(x^*_1,x^*_2)}_{\psi^*}
=\max_{(t_1,t_2)\in\Omega_2}\dfrac{t_1\cdot\mathrm{m} +t_2\cdot\mathrm{m}}{\psi (t_1,t_2)}=\dfrac{\mathrm{m}}{\min\limits_{(t_1,t_2)\in\Omega_2}\psi(t_1,t_2)}=1.
\end{align*}	
Thus, condition \eqref{E3.7-4} is satisfied.
The conclusion is a direct consequence of Theorem~\ref{T3.5}.
\end{proof}	

\begin{proposition}\label{C3.8}
Let $X$ be a Hilbert space, and $\psi\in \pmb{\Psi}_n$.
Then 
\begin{gather*}
\emptyset\ne \text{\rm Sol}(P,\vertiii{\cdot}_\psi)\subset \text{\rm conv}\{v_1,\ldots,v_n\}.
\end{gather*}	
\end{proposition}	

\begin{proof}
Every Hilbert space is a reflexive Banach space. 
By Propositions~\ref{P3.2} and \ref{P4.3}\eqref{P4.3-1}, the solution set is nonempty.
Let $u\in \text{\rm Sol}(P,\vertiii{\cdot}_\psi)$.
By Theorem~\ref{T3.5}, there exist $x^*_1,\ldots,x^*_n\in X^*$ such that
\begin{gather*}
\displaystyle \sum_{i=1}^n x_i^*=0,\;\;
\vertiii{(x^*_1,\ldots,x^*_n)}_{\psi^*}=1,\;\;
\displaystyle \sum_{i=1}^n \langle x_i^*,u-v_i\rangle=\vertiii{(u-v_1,\ldots,u-v_n)}_{\psi}.
\end{gather*}	
In view of the last equality and Proposition~\ref{P3.5+}\eqref{P3.5-2}, 
\begin{gather*}
\ang{x^*_i,u-v_i}=\|x^*_i\|^*\cdot\|u-v_i\|\;\;(i=1,\ldots,n).
\end{gather*}
By the Riesz representation theorem, there exist unique 
vectors $\bx_1,\ldots,\bx_n\in X$ such that
\begin{gather*}
\sum_{i=1}^{n}\bx_i=0,\;\;\|x^*_i\|^*=\|\bx_i\|,\;\;\langle \bx_i,u-v_i\rangle =\|\bx_i\|\cdot\|u-v_i\|\;\;(i=1,\ldots,n).
\end{gather*}	
Thus, there exist scalars $\lambda_i\ge 0$ such that $\lambda:=\sum_{i=1}^{n}\lambda_i>0$ and $\bx_i=\lambda_i(u-v_i)$ $(i=1,\ldots,n)$.
Then $0=\sum_{i=1}^{n}\bx_i=\lambda u-\sum_{i=1}^{n}\lambda_iv_i$, and consequently,
$$u=\frac{\lambda_1}{\lambda}v_1+\ldots+\frac{\lambda_n}{\lambda}v_n\in \text{\rm conv}\{v_1,\ldots,v_n\}.$$
The proof is complete.
\end{proof}	

\begin{remark}
\begin{enumerate}	
\item		
The conclusion of  Proposition~\ref{C3.8} may not be true if $X$ is not a Hilbert space.
Several examples and counterexamples are provided in Sections~\ref{S4}--\ref{S6}.
\item 
Proposition~\ref{C3.8} provides a necessary condition for a norm on a vector space to be characterized by an inner product.
In some particular cases, this necessary condition is also sufficient \cite{MenPak05}. 
\end{enumerate}
\end{remark}

\if{
The next example illustrates the result in Corollary~\ref{C3.8}.
\begin{example}
Let $X:=\R^2$ be equipped with the Euclidean norm $\|\cdot\|$, and $v_1,v_2,v_3$
be three distinct points. 
Consider the classical Fermat–Weber problem, i.e.,
\begin{gather*}
{\rm{minimize }}\;\;f(u):=\|v_1-u\|+\|v_2-u\|+\|v_3-u\|\;\;\text{for all}\;\;
u\in \R^2.
\end{gather*}
If none of the angles of the triangle formed by $v_1,v_2,v_3$ is greater than or equal to $120^\circ$, the unique solution lies inside the triangle. 
Otherwise, if one angle is $120^\circ$ or larger, the solution coincides with the vertex at that angle.
In either case, the solution is unique and lies within the convex hull of the three points.
\end{example}
}\fi

\section{Fermat--Torricelli Problem}\label{S4}
Let $(X,\|\cdot\|)$ be a normed space.
We study a particular case of problem \eqref{P} with $\vertiii{\cdot}:=\vertiii{\cdot}_1$.
The objective function is of the form:
\begin{gather}\label{4.1}
f(u)=\|u-v_1\|+\ldots+\|u-v_n\|\;\;\text{for all}\;\;u\in X.
\end{gather}

The following statement provides optimality conditions and a formula for the solution set of the Fermat–Torricelli problem.
\begin{theorem}\label{T2.7}
A vector $\bar u\in \text{\rm Sol}(P,\vertiii{\cdot}_{1})$ if and only if there exist $x^*_1,\ldots,x^*_n\in X^*$ such that
\begin{gather}\notag
\sum_{i=1}^{n}x^*_i=0,\;\;\max_{1\le i\le n}\|x^*_i\|^*=1,
\\\label{T3.9-1}
\ang{x^*_i,\bar u-v_i}=\|\bar u-v_i\|,\;\;(\|x^*_i\|^*-1)\cdot\|\bar u-v_i\|=0\;\;(i=1,\ldots,n).
\end{gather}
Moreover, if $\bar u\notin \{v_1,\ldots,v_n\}$, then
\begin{gather}\label{T2.7-2}
\text{\rm Sol}(P,\vertiii{\cdot}_{1})=\bigcap_{i=1}^n {\text{\rm \textbf{A}}_{\|\cdot\|}(v_i,x_i^*)},
\end{gather}
where 
\begin{gather}\label{T2.7-30}
\text{\rm \textbf{A}}_{\|\cdot\|}(v_i,x_i^*):=v_i+{\text{\rm T}}_{\|\cdot\|}(x^*_i)=	
\left\{u\in X\mid 
\ang{x^*_i,u-v_i}=\|x^*_i\|^*\cdot \|u-v_i\| \right\}.
\end{gather}	
\end{theorem}

\begin{proof}
\if{
Suppose that $\bar u\in \text{\rm Sol}(P,\vertiii{\cdot}_{1})$.	
By Theorem~\ref{T3.5}, there exist $x^*_1,\ldots,x^*_n\in X^*$ such that conditions \eqref{T3.9-0} are satisfied, and
\begin{gather}\label{T2.7-4}
\sum_{i=1}^{n}\ang{x^*_i,v_i- \bar u}=
\sum_{i=1}^{n}\|v_i-\bar u\|.
\end{gather}	
By Proposition~\ref{P2.6}\;\eqref{P2.6-2}, we have
$\ang{x^*_i,v_i-\bar u}=\|x^*_i\|^*\cdot\|v_i-\bar u\|$ $(i=1,\ldots,n)$.
If $v_{i}-\bar u\ne 0$ for some $i\in\{1,\ldots,n\}$, then it must hold $\|x^*_{i}\|^*=1$. 
Indeed, if $\|x^*_i\|^*<1$ then $\ang{x^*_{i},v_{i}-\bar u}<\|v_{i}-\bar u\|$, and consequently,
\begin{gather*}
\sum_{i=1}^{n}\ang{x^*_i,v_i- \bar u}\le\sum_{i=1}^{n}\|x^*_i\|^*\cdot\|v_i-\bar u\|<
\sum_{i=1}^{n}\|v_i-\bar u\|,
\end{gather*}	
which contradicts \eqref{T2.7-4}.
Thus, condition \eqref{T3.9-1} holds true.
Conversely, suppose that there exist $x^*_1,\ldots,x^*_n\in X^*$ such that conditions \eqref{T3.9-0} and \eqref{T3.9-1} are satisfied.
Then condition \eqref{T2.7-4} is satisfied.
By Theorem~\ref{T3.5}, we have $\bar u\in \text{\rm Sol}(P,\vertiii{\cdot}_{1})$.
}\fi
The first part is a direct consequence of Proposition~\ref{P3.5+}\eqref{P3.6-2} and Theorem~\ref{T3.5}.
Suppose that $\bar u\notin \{v_1,\ldots,v_n\}$.
By \eqref{T3.9-1}, 
\begin{gather}\label{T3.9-10}
\|x^*_1\|^*=\ldots=\|x^*_n\|^*=1.
\end{gather}	
Let $u\in \text{\rm Sol}(P,\vertiii{\cdot}_{1})$.
In view of \eqref{P2.5-1} with $\vertiii{\cdot}:=\vertiii{\cdot}_1$, we have
\begin{gather}\label{T3.9-3}
\sum_{i=1}^{n}\ang{x^*_i,u-v_i}=\sum_{i=1}^{n}\|u-v_i\|.
\end{gather}
Then
\begin{gather}\label{T2.7-3}
\ang{x^*_i,u-v_i}=\|u-v_i\|\;\;
(i=1,\ldots,n).
\end{gather}
Thus,
$u\in\bigcap_{i=1}^n\text{\rm \textbf{A}}_{\|\cdot\|}(v_i,x^*_i)$.
Conversely, suppose that $u\in\bigcap_{i=1}^n\text{\rm \textbf{A}}_{\|\cdot\|}(v_i,x^*_i)$.
By \eqref{T3.9-10} and \eqref{T2.7-3}, condition \eqref{T3.9-3} holds true.
By Theorem~\ref{T3.5}, we have $u\in \text{\rm Sol}(P,\vertiii{\cdot}_{1})$.
\end{proof}		

\begin{remark}
\begin{enumerate}
\item
It follows from \eqref{T3.9-1} that $x^*_i\in\partial\|\cdot\|(\bar u-v_i)$ if $\bar u\ne v_i$.  
\item 
A finite dimensional Minkowski space version of Theorem~\ref{T2.7} can be found in \cite[Theorems~3.1 \& 3.2]{MarSwaWei02}.
When $X:=\R^n$ is equipped with the Euclidean norm, Theorem~\ref{T2.7} improves \cite[Reformulation~18.4\;(i)]{BolMarSol99}.
\item 
A result similar to Theorem~\ref{T2.7} in the setting of a sum of convex functions in finite dimensional spaces can be found in \cite[Lemma~3.1]{DurMic85}.
\end{enumerate}
\end{remark}	

The following examples illustrate Theorem~\ref{T2.7}.
\begin{example}\label{E4.4}
Let $\R^2$ be equipped with  some norm $\|\cdot\|$, and
\begin{gather}\label{E4.4.1}
v:=(0,0),\;\;w:=(2,0),\;\;\bar u:=(1,0).
\end{gather}	
The objective function \eqref{4.1} is of the form
\begin{gather}\label{E4.4-0}
f(u)=\|u-v\|+\|u-w\|
\end{gather}	
for all $u\in\R^2.$
By Example~\ref{E3.7}, we have $\bar u\in\text{\rm Sol}(P)$.
Define $v^*:=(1,0)$ and $w^*:=(-1,0)$.
Then $v^*+w^*=0$.
\begin{enumerate}
\item\label{E4.4-1}
Suppose that $\|\cdot\|:=\|\cdot\|_\infty$.
The dual norm is $\|\cdot\|_1$. 
We have 
\begin{center}
$\|v^*\|_1=\|w^*\|_1=1,\;\ang{v^*,\bar u-v}=\|\bar u-v\|_\infty=1,\;\ang{w^*,\bar u-w}=\|\bar u-w\|_\infty=1$.
\end{center}
By Lemma~\ref{L1.1}\eqref{L1.1-4},
\begin{align}\notag
\text{\rm \textbf{A}}_{\|\cdot\|_\infty}(v,v^*)
&
=\left\{(u_1,u_2)\mid 
|u_1|=\max\{|u_1|,|u_2|\},\; u_1\ge 0
\right\}\\ \label{E4.4-7}
&=\left\{(u_1,u_2)\mid 
u_1\ge |u_2|
\right\},\\ \notag
\text{\rm \textbf{A}}_{\|\cdot\|_\infty}(w,w^*)
&
=\left\{(u_1,u_2)\mid 
|u_1-2|=\max\{|u_1-2|,|u_2|\},\;u_1\le 2
\right\}\\ \label{E4.4-8}
&=\left\{(u_1,u_2)\mid 2-u_1\ge |u_2|
	\right\}.
\end{align}	
By \eqref{T2.7-2}, 
\begin{gather*}
	\text{\rm Sol}(P)=\text{\rm \textbf{A}}_{\|\cdot\|_\infty}(v,v^*)\cap \text{\rm \textbf{A}}_{\|\cdot\|_\infty}(w,w^*)=\left\{(u_1,u_2)\mid 
	u_1\ge |u_2|,\;2-u_1\ge |u_2|
	\right\}.
\end{gather*}
\begin{figure}[H]
\centering
\includegraphics[width=0.5\linewidth]{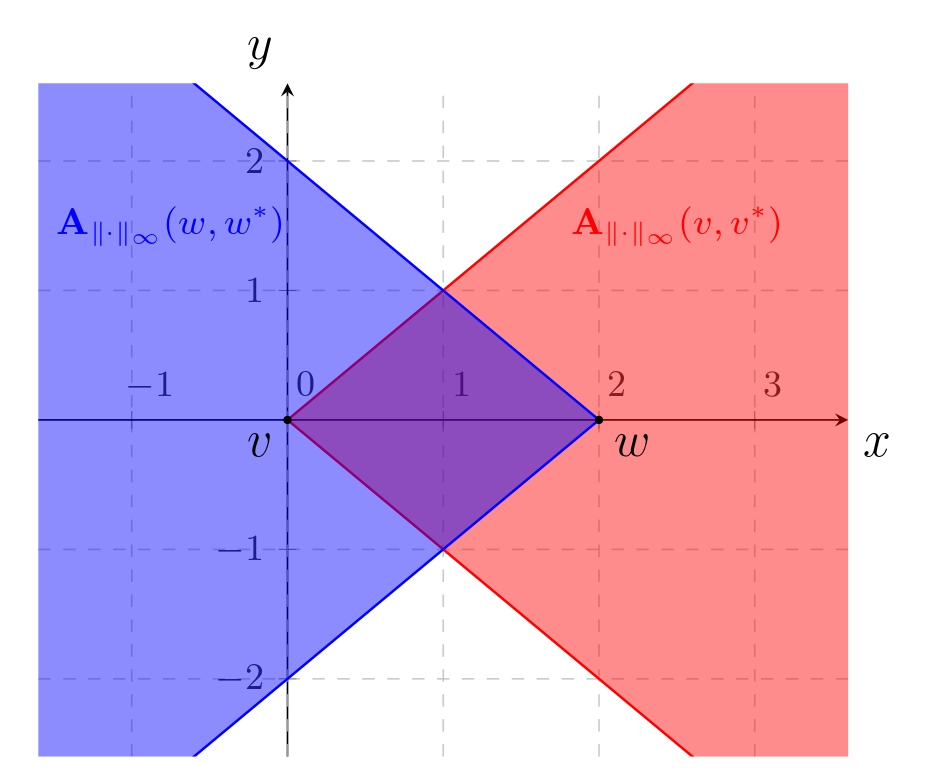}
\caption{Sets $\text{\rm \textbf{A}}_{\|\cdot\|_\infty}(v,v^*)$ and $\text{\rm \textbf{A}}_{\|\cdot\|_\infty}(w,w^*)$}
\label{fig1}
\end{figure}
\if{
\item\label{E4.4-2}
Suppose that $\|\cdot\|:=\|\cdot\|_1$.
The dual norm of $\|\cdot\|_1$ is $\|\cdot\|_\infty$. 
Then
\begin{gather*}
\max\{\|v^*\|_\infty,\|w^*\|_\infty\}=\max\{1,1\}=1,\\
\ang{v^*,v-\bar u}=\|v-\bar u\|_1=1,\;\;\ang{w^*,w-\bar u}=\|w-\bar u\|_1=1.
\end{gather*}	
By Lemma~\ref{L1.1}\;\eqref{L1.1-3},
\begin{align}\label{E4.4-5}
\text{\rm \textbf{A}}(v,v^*)
&=\left\{(u_1,u_2)\in\R^2\mid u_1\ge 0,\; u_2= 0\right\},\\ \label{E4.4-6}
\text{\rm \textbf{A}}(w,w^*)
&=\left\{(u_1,u_2)\in\R^2\mid u_1\le 2,\;u_2= 0\right\}.
\end{align}	
By \eqref{T2.7-2}, we have
$$\text{\rm Sol}(P)=\text{\rm \textbf{A}}(v,v^*)\cap \text{\rm \textbf{A}}(w,w^*)=\left\{(u_1,u_2)\in\R^2\mid 0\le u_1\le 2,\;u_2=0\right\}.$$
}\fi
\item\label{E4.4-3}
Suppose that $\|\cdot\|:=\|\cdot\|_p$ with $1\le p<\infty$.
The dual norm  is $\|\cdot\|_q$ with $q\in(1,\infty]$ satisfying
$\frac{1}{p}+\frac{1}{q}=1$.
We have  
\begin{center}
$\|v^*\|_q=\|w^*\|_q=1,\;\ang{v^*,\bar u-v}=\|\bar u-v\|_p=1,\;\ang{w^*,\bar u-w}=\|\bar u-w\|_p=1$.
\end{center}
By Lemma~\ref{L1.1}\eqref{L1.1-3} and \eqref{L1.1-2},
\begin{align}\label{E4.4-9}
\text{\rm \textbf{A}}_{\|\cdot\|_p}(v,v^*)
=\left\{(\lambda,0)\mid\lambda\ge 0 \right\},\;\; 
\text{\rm \textbf{A}}_{\|\cdot\|_p}(w,w^*)
=\left\{(\lambda,0)\mid \lambda\le 2\right\}.
\end{align}	
By \eqref{T2.7-2}, we have
$\text{\rm Sol}(P)=\text{\rm \textbf{A}}_{\|\cdot\|_p}(v,v^*)\cap \text{\rm \textbf{A}}_{\|\cdot\|_p}(w,w^*)=\left\{(\lambda,0)\mid 0\le \lambda\le 2
\right\}.$
\end{enumerate}
\begin{figure}[H]
	\centering
	\includegraphics[width=0.9\linewidth]{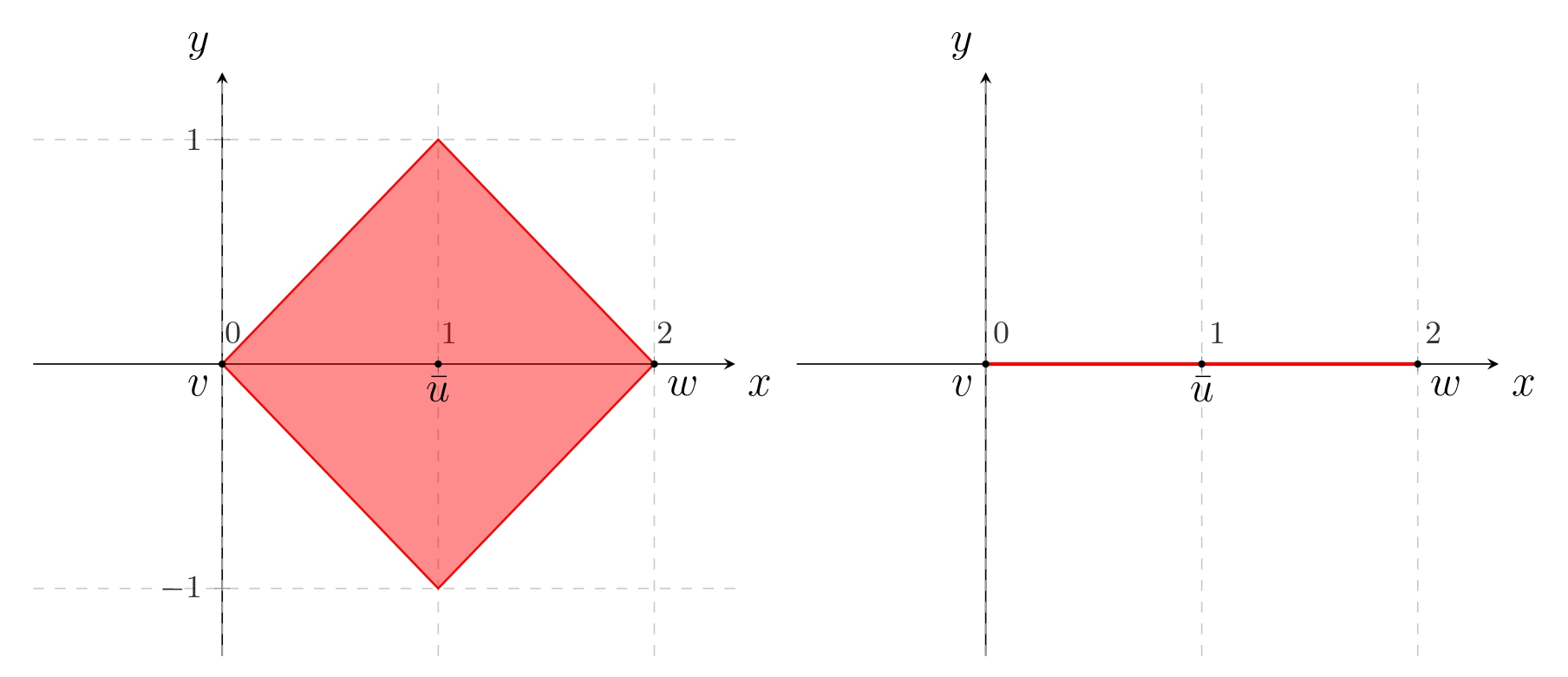}
   {\scriptsize
	Example~\ref{E4.4}\eqref{E4.4-1} \hspace{6cm} 
	Example~\ref{E4.4}\eqref{E4.4-3}
}
	\caption{Solution sets in Example~\ref{E4.4}}
	\label{fig2}
\end{figure}
\end{example}

\begin{example}\label{E4.3}
Let $X:=L^2([0,1])$.
It is a Hilbert space  with the inner product	
\begin{gather*}
\ang{x,y}:=\displaystyle\int_{0}^1 x(t)y(t)dt\;\;\text{for all}\;\;
x,y\in X.
\end{gather*}	
Let 
\begin{gather}\label{E5.4-1}
	v(t):=t\;\;\text{for all}\;\;t\in[0,1],\;\;w:=-v,\;\;\bar u:=0\in X.
\end{gather}	
The objective function is of the form \eqref{E4.4-0} for all $u\in X$.
By Example~\ref{E3.7}, we have $\bar u\in\text{\rm Sol}(P)$.
By Theorem~\ref{T2.7}, there
exist $v^*,w^*\in X^*$ such that
\begin{gather*}
v^*+w^*=0,\;
\max\{\|v^*\|^*,\|w^*\|^*\}=1,
\ang{v^*,\bar u-v}=\|\bar u-v\|,\;
\ang{w^*,\bar u-w}=\|\bar u-w\|.
\end{gather*}	
By the Riesz representation theorem, there exist unique vectors
$\bar v:=-\sqrt 3 v$ and $\bar w:=\sqrt 3 v$ 
such that the above conditions are satisfied with $\bar v$ and $\bar w$ in place of $v^*$ and $w^*$, respectively.
Indeed,
\begin{gather*}
\bar v+\bar w=0,\;\|\bar v\|=\|\bar w\|=1,\; 
\ang{\bar v,\bar u-v}=\|\bar u-v\|=1/\sqrt 3,\;
\ang{\bar w,\bar u-w}=\|\bar u-w\|=1/\sqrt 3.
\end{gather*}	
By Lemma~\ref{L1.1}\eqref{L1.1-1}, 
\begin{align*}
\text{\rm \textbf{A}}_{\|\cdot\|}(v,v^*)
&=\left\{v+\lambda \bar v\mid\lambda\ge 0 \right\}
=\left\{\lambda v\mid \lambda\le 1\right\},\\
\text{\rm \textbf{A}}_{\|\cdot\|}(w,w^*)
&=\left\{w+\lambda \bar w\mid\lambda\ge 0 \right\}
=\left\{\lambda v\mid \lambda\ge -1\right\}.
\end{align*}	
By \eqref{T2.7-2}, we have
$\text{\rm Sol}(P)=\text{\rm \textbf{A}}_{\|\cdot\|}(v,v^*)\cap \text{\rm \textbf{A}}_{\|\cdot\|}(w,w^*)=\left\{\lambda v \mid -1\le\lambda\le 1
\right\}.$
\end{example}

\section{Chebyshev Centre Problem}\label{S5}
Let $(X,\|\cdot\|)$ be a normed space.
We study a particular case of problem \eqref{P} with $\vertiii{\cdot}:=\vertiii{\cdot}_\infty$.
The  objective function is of the form:
\begin{gather}\label{5.1}
f(u)=\max_{1\le i\le n}\|u-v_i\|\;\;\text{for all}\;\;u\in X.
\end{gather}	
\if{
Define the active index set:
\begin{gather*}
	I(u):=\Big\{i\in\{1,\ldots,n\}\mid \|u-v_i\|=\max_{1\le j\le n}\|u-v_j\|\Big\}\;\;\text{for all}\;\;u\in X.
\end{gather*}	
}\fi

The following statement provides optimality conditions and a formula for the solution set of the Chebyshev centre problem.
\begin{theorem}\label{T5.1}
A vector $\bar u\in \text{\rm Sol}(P,\vertiii{\cdot}_\infty)$ if and only if there exist $x^*_1,\ldots,x^*_n\in X^*$ such that
\begin{gather}\label{T2.9-1}
\sum_{i=1}^{n}x^*_i=0,\;\;\sum_{i=1}^{n}\|x^*_i\|^*=1,
\\\notag
\ang{x^*_i,\bar u-v_i}=\|x^*_i\|^*\cdot\|\bar u-v_i\|,\;\;
\left(\|\bar u-v_i\|-\max_{1\le i\le n}\|\bar u-v_i\|\right)\cdot\|x^*_i\|^*=0\;\;(i=1,\ldots,n).
\end{gather}
Moreover, 
\begin{gather}\label{T2.9-3}
\text{\rm Sol}(P,\vertiii{\cdot}_\infty)=\bigcap_{x^*_i\ne 0} \text{ \rm \textbf{B}}(v_i,x_i^*)
\end{gather}
with
\begin{gather*}
\text{ \rm \textbf{B}}(v_i,x_i^*) :=\left\{u\in X\mid \langle x_i^*, u-v_i\rangle=\|x^*_i\|^*\cdot\max_{1\le j\le n}\|u-v_j\|\right\}.
\end{gather*}
\end{theorem}

\begin{proof}
\if{	
The dual norm of $\vertiii{\cdot}_\infty$ is $\vertiii{\cdot}_1$. 	
Suppose that $\bar u\in\text{\rm Sol}(P,\vertiii{\cdot}_\infty)$.
By Theorem~\ref{T3.5}, there exist $x^*_1,\ldots,x^*_n\in X^*$ such that conditions \eqref{T2.9-1} are satisfied, and
\begin{gather}\label{T2.9-4}
\sum_{i=1}^{n}\langle x^*_i,v_i-\bar u \rangle=\max\limits_{1\le j\le n}\|v_j-\bar u\|.
\end{gather}
By Proposition~\ref{P2.6}\;\eqref{P2.6-2}, we have
$\ang{x^*_i,v_i-\bar u}=\|x^*_i\|^*\cdot\|v_i-\bar u\|$ $(i=1,\ldots,n)$.
If $x^*_i\ne 0$ for some $i\notin I(\bar u)$, then
\begin{gather*}
\sum_{i=1}^{n}\ang{x^*_i,v_i- \bar u}\le\sum_{i=1}^{n}\|x^*_i\|^*\cdot\|v_i-\bar u\|<\max\limits_{1\le j\le n}\|v_j-\bar u\|,
\end{gather*}	
which contradicts \eqref{T2.9-4}.
Thus, condition \eqref{T2.9-2} holds true.
Conversely, suppose that there exist $x^*_1,\ldots,x^*_n\in X^*$ such that conditions \eqref{T2.9-1} and \eqref{T2.9-2} are satisfied.
Then condition \eqref{T2.9-4} is satisfied.
By Theorem~\ref{T3.5}, we have $\bar u\in \text{\rm Sol}(P,\vertiii{\cdot}_{\infty})$.
}\fi
The first part  is a direct consequence of Proposition~\ref{P3.5+}\eqref{P3.6-1} and Theorem~\ref{T3.5}.
We now prove \eqref{T2.9-3}.
Let $u\in \text{\rm Sol}(P,\vertiii{\cdot}_\infty)$.
In view of \eqref{P2.5-1} with $\vertiii{\cdot}:=\vertiii{\cdot}_\infty$, we have
\begin{gather}\label{T2.9-5}
\sum_{i=1}^{n}\ang{x^*_i,u-v_i}=M,
\end{gather}
where $M:=\max_{1\le j\le n} \|u-v_j\|$.
Then 
\begin{gather}\label{T2.9-6}
\langle x_i^*,u-v_i\rangle=M\cdot\|x^*_i\|^*
\end{gather}	
for $x^*_i\ne 0$.
Indeed, if $\langle x_i^*, u-v_i\rangle<M\cdot\|x^*_i\|^*,$
then,  by the second equality in \eqref{T2.9-1}, we have $\sum_{i=1}^{n}\ang{x^*_i,u-v_i}<M$
contradicting \eqref{T2.9-5}.
Thus, $u\in\bigcap_{i=1}^n\text{ \rm \textbf{B}}(v_i,x^*_i)$.
Conversely, suppose that $u\in \bigcap_{i=1}^n \text{ \rm \textbf{B}}(v_i,x_i^*)$. 
By \eqref{T2.9-6} and the second equality in \eqref{T2.9-1}, condition \eqref{T2.9-5} is satisfied.
By Theorem~\ref{T3.5}, we have $u\in \text{\rm Sol}(P,\vertiii{\cdot}_\infty)$.
\end{proof}

\begin{remark}\label{R5.2}
\begin{enumerate}
\item\label{R5.2-1}
In view of conditions \eqref{T2.9-1}, at least two dual vectors are nonzero.
\item\label{R5.2-2}
The set $\text{ \rm \textbf{B}}(v_i,x^*_i)$ is  the intersection of $\text{\rm \textbf{A}}_{\|\cdot\|}(v_i,x^*_i)$ given by \eqref{T2.7-30} and
the set of the farthest Voronoi cells \cite{GobMarTod19} of $v_i$ with respect to the set $\{v_1,\ldots,v_n\}$ and the norm $\|\cdot\|$ given by
\begin{gather}\label{R6.2-1}
\text{\rm \textbf{FV}}_{\|\cdot\|}(v_i\mid v_1,\ldots,v_n):=\left\{u\in X\mid \|u-v_i\|=\max_{1\le j\le n}\|u-v_j\|\right\}.
\end{gather}
If $X$ is a Hilbert space, then
\begin{gather*}
\text{\rm \textbf{FV}}_{\|\cdot\|}(v_i\mid v_1,\ldots,v_n)=\bigcap_{j=1}^n\left\{u\in X\mid \ang{u,v_j-v_i}\ge \dfrac{\|v_j\|^2-\|v_i\|^2 }{2}\right\};
\end{gather*}		
 see, for instance, \cite[p. 309]{GobMarTod19}.
 \item 
The solution set formula \eqref{T2.9-3} provides a partial answer to the open problem \cite[Problem~7]{Hoa19}.
\item 
To the best of our knowledge, although the dual characterizations of the Chebyshev centre in Theorem~\ref{T5.1} follow directly from standard tools of convex and functional analysis, they have not been explicitly formulated in the literature.
In particular, the explicit description of the solution set in \eqref{T2.9-3} provides a  way to obtain all solutions from any given one together with its associated dual vectors.
\end{enumerate}
\end{remark}	
\if{
The next result presents explicit formulas for the set of farthest Voronoi cells in specific cases of the norm on 
$X$.
\begin{lemma}\label{L5.3}
Let $X$ be a normed vector space, and $v_1,\ldots,v_n\in X$.
\begin{enumerate}
\item\label{L5.3-1}
Suppose that $X$ is a Hilbert space.
Then
\begin{gather}\label{P2.13-4}
\text{\rm \textbf{FV}}(v_i\mid v_1,\ldots,v_n)=\bigcap_{j=1}^n\left\{u\in X\mid \ang{u,v_j-v_i}\ge \dfrac{\|v_j\|^2-\|v_i\|^2 }{2}\right\}\;\;(i=1,\ldots,n).
\end{gather}		
\item\label{L5.3-2}
Suppose that $X:=\R^n$, $\|\cdot\|:=\|\cdot\|_1$, and
$v_i:=(v_{i1},\ldots,v_{in})$ $(i=1,\ldots,n)$ .
Then
\begin{gather*}
\text{\rm \textbf{FV}}(v_i\mid v_1,\ldots,v_n)=\bigcap_{j=1}^n\left\{u:=(u_1,\ldots,u_n)\in\R^n\mid \sum_{k=1}^{n}\left(|v_{ik}-u_k|-|v_{jk}-u_k|\right)\ge 0\right\},
\end{gather*}	
\end{enumerate}
\end{lemma}
\begin{proof}
Assertion \eqref{L5.3-1}  is a consequence of the standard fact that
$\|v_i-u\|^2 \ge \|v_j-u\|^2$ if and only if $\ang{u,v_j-v_i}\ge \frac{\|v_j\|^2-\|v_i\|^2 }{2}$ for all $i,j=1,\ldots,n$, while assertion \eqref{L5.3-2} follows from the definition \eqref{FV}.
\end{proof}	

\begin{remark}
\begin{enumerate}
\item		
The finite dimentional version of Lemma~\ref{L5.3}\;\eqref{L5.3-1} can be find in \cite[p. 309]{GobMarTod19}.
\item
Let $i,j\in\{1,\ldots,n\}$ and $u:=(u_1,\ldots,u_n)$.
Set $\gamma_{ij,k}(u_k):=|v_{ik}-u_k|-|v_{jk}-u_k|$ $(k=1,\ldots,n)$.
Let $k=1,\ldots,n$.
Then
\begin{gather*}
\gamma_{ij,k}(u_k)=\begin{cases}
v_{ik}-v_{jk}& \text{if } u_k\le \min\{v_{ik},v_{jk}\}\\
v_{jk}-v_{ik} & \text{if }
u_k\ge \max\{v_{ik},v_{jk}\},\\
2u_k-v_{ik}-v_{jk} & \text{if }
v_{ik}<v_{jk}\;\;\text{and}\;\;
v_{ik}<u_k<v_{jk},\\
v_{ik}+v_{jk}-2u_k & \text{if }
v_{jk}<v_{ik}\;\;\text{and}\;\;
v_{jk}<u_k<v_{ik}.
\end{cases}
\end{gather*}
\end{enumerate}	
\end{remark}
}\fi

The next example computes the set \eqref{R6.2-1} for  $p$-norms with $p\in[1,\infty]$.
\begin{example}\label{E5.3}
Let $\R^2$ be equipped with  some norm $\|\cdot\|$, $v:=(0,0)$ and $w:=(2,0)$.
\begin{enumerate}
\item\label{E5.3-2}
Let $\|\cdot\|:=\|\cdot\|_\infty$.
By \eqref{R6.2-1},
\begin{align*}
\text{\rm \textbf{FV}}_{\|\cdot\|_\infty}(v\mid v,w)
&=\left\{(u_1,u_2)\mid \max\{|u_1|,|u_2|\}\ge\max\{|u_1-2|,|u_2|\}\right\}\\
&=\left\{(u_1,u_2)\mid u_1\ge 1\right\}\cup \left\{(u_1,u_2)\mid |u_1-2|\le|u_2|\right\},\\
\text{\rm \textbf{FV}}_{\|\cdot\|_\infty}(w\mid v,w)
&=\left\{(u_1,u_2)\mid \max\{|u_1-2|,|u_2|\}\ge\max\{|u_1|,|u_2|\}\right\}\\
&=\left\{(u_1,u_2)\mid u_1\le 1\right\}\cup \left\{(u_1,u_2)\mid |u_1|\le|u_2|\right\}.
\end{align*}	
\item\label{E5.3-1}
Let $\|\cdot\|:=\|\cdot\|_p$ with $p\in[1,\infty)$.
By \eqref{R6.2-1},
\begin{align*}
\text{\rm \textbf{FV}}_{\|\cdot\|_p}(v\mid v,w)
&=\left\{(u_1,u_2)\mid |u_1|\ge|u_1-2|\right\}
=\left\{(u_1,u_2)\mid u_1\ge 1\right\},\\ 
\text{\rm \textbf{FV}}_{\|\cdot\|_p}(w\mid v,w)
&=\left\{(u_1,u_2)\mid |u_1-2|\ge|u_1|\right\}
=\left\{(u_1,u_2)\mid u_1\le 1\right\}.
\end{align*}	
\end{enumerate}		
\begin{figure}[H]
\centering
\includegraphics[width=0.9\linewidth]{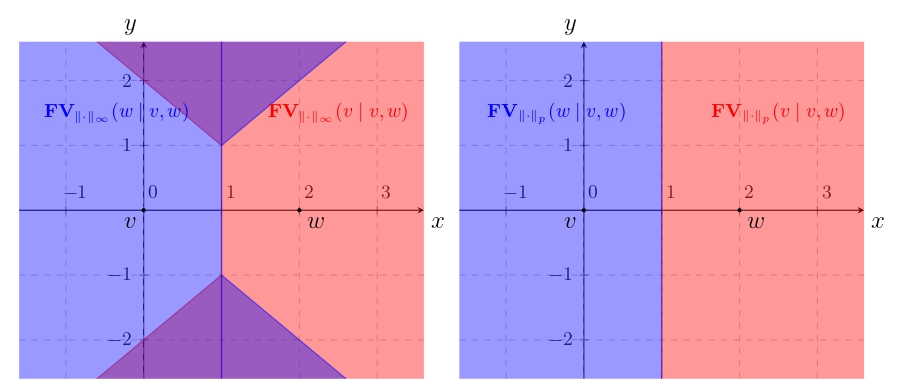}
{\scriptsize
Example~\ref{E5.3}\eqref{E5.3-2} \hspace{4cm} 
Example~\ref{E5.3}\eqref{E5.3-1}
}
\caption{Example~\ref{E5.3}}
\label{fig3}
\end{figure}
\end{example}	

The following examples illustrate  Theorem~\ref{T5.1}.
\begin{example}\label{E5.5}
Let $\R^2$ be equipped with some norm $\|\cdot\|$, vectors $v,w$ and $\bar u$ be given by \eqref{E4.4.1}.
The objective function \eqref{5.1} is of the form
\begin{gather}\label{E5.5-0}
f(u):=\max\{\|u-v\|,\|u-w\|\}
\end{gather}	
for all $u\in\R^2.$
By Example~\ref{E3.7}, we have $\bar u\in\text{\rm Sol}(P)$.
Define $v^*:=(1/2,0)$ and $w^*:=(-1/2,0)$.
Then $v^*+w^*=0$.
\begin{enumerate}
\item\label{E5.5-1}
Suppose that $\|\cdot\|:=\|\cdot\|_\infty$.
We have $\|v^*\|_1=\|w^*\|_1=1/2$, and
$$\ang{v^*,\bar u-v}=\|v^*\|_1\cdot\|\bar u-v\|_\infty=1/2,\;\ang{w^*,\bar u-w}=\|w^*\|_1\cdot\|\bar u-w\|_\infty=1/2.$$
By \eqref{E4.4-7}, \eqref{E4.4-8}, Example~\ref{E5.3}\eqref{E5.3-2} and Remark~\ref{R5.2}\eqref{R5.2-2},
\begin{align*}
\text{ \rm \textbf{B}}_{\|\cdot\|_\infty}(v,v^*)
&=\text{\rm \textbf{A}}_{\|\cdot\|_\infty}(v,v^*)\cap \text{\rm \textbf{FV}}_{\|\cdot\|_\infty}(v\mid v,w)
=
\left\{(u_1,u_2)\mid u_1\ge 1,\;u_1\ge |u_2|\right\},\\
\text{ \rm \textbf{B}}_{\|\cdot\|_\infty}(w,w^*)
&=\text{\rm \textbf{A}}_{\|\cdot\|_\infty}(w,w^*)\cap \text{\rm \textbf{FV}}_{\|\cdot\|_\infty}(w\mid v,w)
=\left\{(u_1,u_2)\mid u_1\le 1,\; 2-u_1\ge|u_2|\right\}.
\end{align*}	
By \eqref{T2.9-3}, we have
$\text{\rm Sol}(P)=\text{ \rm \textbf{B}}_{\|\cdot\|_\infty}(v,v^*)\cap \text{ \rm \textbf{B}}_{\|\cdot\|_\infty}(w,w^*)=\left\{(1,\lambda)\mid |\lambda|\le 1\right\}.$
\item\label{E5.5-2}
Suppose that $\|\cdot\|:=\|\cdot\|_p$ with $1\le p<\infty$.
The dual norm  is $\|\cdot\|_q$ with $q\in(1,\infty]$ satisfying
$\frac{1}{p}+\frac{1}{q}=1$.
We have $\|v^*\|_q=\|w^*\|_q=1/2$, and
$$\ang{v^*,\bar u-v}=\|v^*\|_q\cdot \|\bar u-v\|_p=1/2,\;\ang{w^*,\bar u-w}=\|w^*\|_q\cdot\|\bar u-w\|_p=1/2.$$
In view of \eqref{E4.4-9},  Example~\ref{E5.3}\eqref{E5.3-1} and
Remark~\ref{R5.2}\eqref{R5.2-2}, 
\begin{align*}
\text{ \rm \textbf{B}}_{\|\cdot\|_p}(v,v^*)
&=\text{\rm \textbf{A}}_{\|\cdot\|_p}(v,v^*)\cap \text{\rm \textbf{FV}}_{\|\cdot\|_p}(v\mid v,w)=
\left\{(\lambda,0)\mid \lambda\ge 1\right\},\\
\text{ \rm \textbf{B}}_{\|\cdot\|_p}(w,w^*)
&=\text{\rm \textbf{A}}_{\|\cdot\|_p}(w,w^*)\cap \text{\rm \textbf{FV}}_{\|\cdot\|_p}(w\mid v,w)=\left\{(\lambda,0)\mid \lambda\le 1\right\}.
\end{align*}	
By \eqref{T2.9-3}, we have $\text{\rm Sol}(P)=\text{ \rm \textbf{B}}_{\|\cdot\|_p}(v,v^*)\cap \text{ \rm \textbf{B}}_{\|\cdot\|_p}(w,w^*)=\{\bar u\}.$
\end{enumerate}
\begin{figure}[H]
\centering
\includegraphics[width=0.9\linewidth]{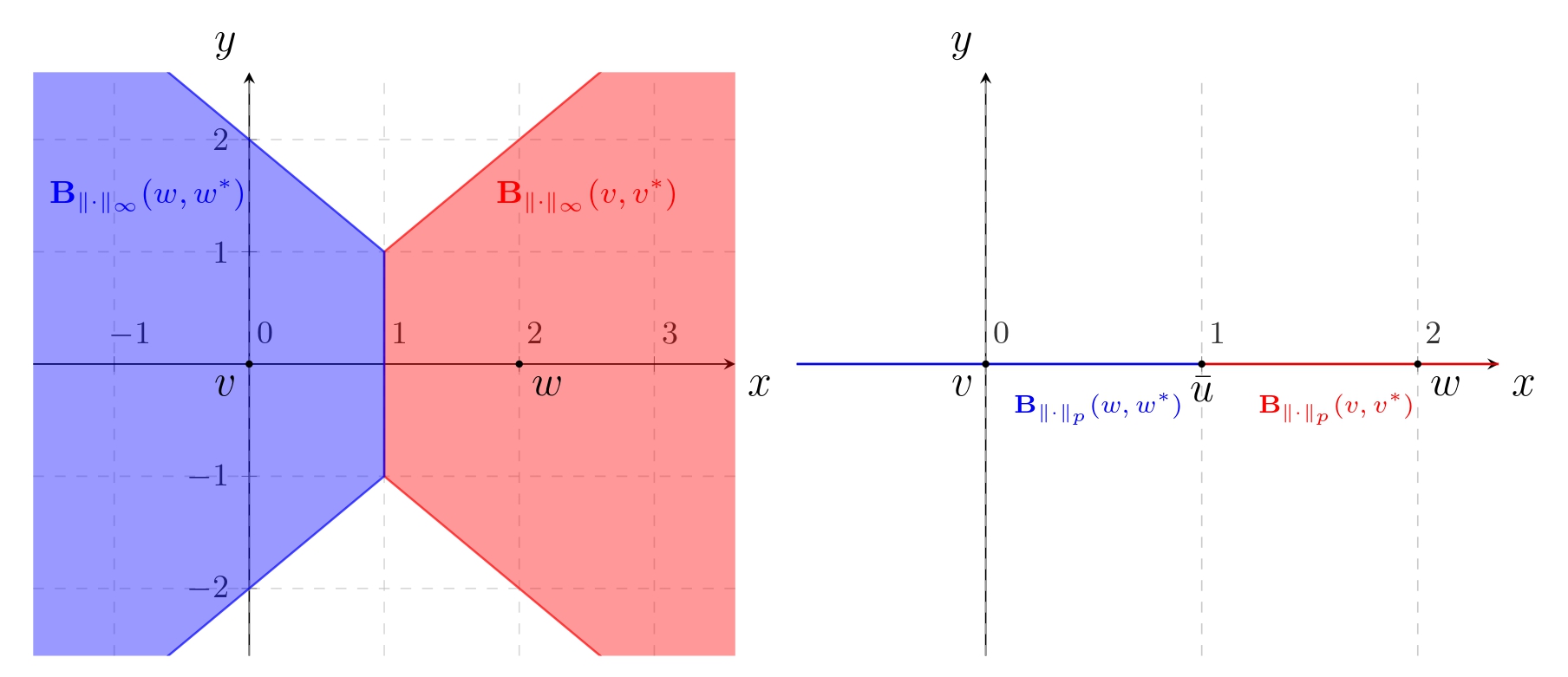}
{\scriptsize
Example~\ref{E5.5}\eqref{E5.5-1} \hspace{4cm} 
Example~\ref{E5.3}\eqref{E5.3-1}
}
\caption{Example~\ref{E5.5}}
\label{fig4}
\end{figure}
\end{example}

\begin{example}
Let $X$ be the Hilbert space in Example~\ref{E4.3}, and vectors $v,w,\bar u$ be given by \eqref{E5.4-1}.
The objective function is of the form
\eqref{E5.5-0} for all $u\in X$.
By Example~\ref{E3.7}, we have $\bar u\in\text{\rm Sol}(P)$.
By Theorem~\ref{T5.1}, there exist $v^*,w^*\in X^*$ such that
\begin{gather*}
v^*+w^*=0,\;
\|v^*\|^*+\|w^*\|^*=1,\;
\ang{v^*,\bar u-v}=\|v^*\|^*\cdot\|\bar u-v\|,\;
\ang{w^*,\bar u-w}=\|w^*\|^*\cdot\|\bar u-w\|.
\end{gather*}	
By the Riesz representation theorem, there exist unique vectors
$\bar v:=-\frac{\sqrt 3}{2}v$ and $\bar w:=\frac{\sqrt 3}{2} v$  such that the above conditions are satisfied with $\bar v$ and $\bar w$ in place of $v^*$ and $w^*$, respectively.
Indeed, 
\begin{gather*}
\bar v+\bar w=0,\;\|\bar v\|=\|\bar w\|=1/2,\\
\ang{\bar v,\bar u-v}=\|\bar v\|\cdot\|\bar u-v\|=\frac{1}{
2\sqrt 3},\;\ang{\bar w,\bar u-w}=
\|\bar w\|\cdot\|\bar u-w\|=\frac{1}{
2\sqrt 3}.
\end{gather*}	
By Lemma~\ref{L1.1}\eqref{L1.1-1}, 
\begin{align*}
	\text{\rm \textbf{A}}_{\|\cdot\|}(v,v^*)
	&=\left\{v+\lambda \bar v\mid\lambda\ge 0 \right\}
	=\left\{\lambda v\mid \lambda\le 1\right\},\\
	\text{\rm \textbf{A}}_{\|\cdot\|}(w,w^*)
	&=\left\{w+\lambda \bar w\mid\lambda\ge 0 \right\}
	=\left\{\lambda v\mid \lambda\ge -1\right\}.
\end{align*}	
By \eqref{R6.2-1},
\begin{align*}
\text{\rm \textbf{FV}}_{\|\cdot\|}(v\mid v,w)
&=\left\{u\in X\mid \|u-v\|\ge\|u-w\|
\right\}=\left\{u\in X\mid \ang{u,v}\ge 0
\right\},\\
\text{\rm \textbf{FV}}_{\|\cdot\|}(w\mid v,w)
&=\left\{u\in X\mid  \|u-w\|\ge\|u-v\|\right\}=\left\{u\in X\mid \ang{u,v}\le 0
\right\}.
\end{align*}	
By Remark~\ref{R5.2}\eqref{R5.2-2}, 
\begin{align*}
\text{ \rm \textbf{B}}_{\|\cdot\|}(v,v^*)
&=\text{\rm \textbf{A}}_{\|\cdot\|}(v,v^*)\cap \text{\rm \textbf{FV}}_{\|\cdot\|}(v\mid v,w)=
\left\{\lambda v \mid 0\le \lambda\le 1\right\},\\
\text{ \rm \textbf{B}}_{\|\cdot\|}(w,w^*)
&=\text{\rm \textbf{A}}_{\|\cdot\|}(w,w^*)\cap \text{\rm \textbf{FV}}_{\|\cdot\|}(w\mid v,w)=\left\{\lambda v \mid -1\le \lambda\le 0\right\}.
\end{align*}	
By \eqref{T2.9-3}, we have
$\text{\rm Sol}(P)=\text{ \rm \textbf{B}}_{\|\cdot\|}(v,v^*)\cap \text{ \rm \textbf{B}}_{\|\cdot\|}(w,w^*)=\{\bar u\}.$
\end{example}	

\section{$p$-Fermat--Torricelli Problem}\label{S6}
Let $(X,\|\cdot\|)$ be a normed space.
We study a particular case of problem \eqref{P} with $\vertiii{\cdot}:=\vertiii{\cdot}_p$ with $p\in(1,\infty)$.
The  objective function is of the form:
\begin{gather}\label{6.1}
f(u)=\left(\|u-v_1\|^p+\ldots+\|u-v_n\|^p\right)^{\frac{1}{p}}  \;\;\text{for all}\;\;u\in X.
\end{gather}	

The following statement provides optimality conditions and a formula for the solution set of the $p$-Fermat-Torricelli problem.
\begin{theorem}\label{T3.16}
A vector $\bar u\in \text{\rm Sol}(P,\vertiii{\cdot}_p)$ if and only if there exist $x^*_1,\ldots,x^*_n\in X^*$ such that
\begin{gather}\notag
\sum_{i=1}^{n}x^*_i=0,\;\;\sum_{i=1}^{n}\|x^*_i\|^{*q}=1,\\ \label{T3.16-2}
\ang{x^*_i,\bar u-v_i}=\|x^*_i\|^*\cdot\|\bar u-v_i\|,\;\;
\|x^*_i\|^{*q}=\dfrac{\|\bar u-v_i\|^p}{\sum_{i=1}^{n}\|\bar u-v_i\|^p}\;\;(i=1,\ldots,n).
\end{gather}
Moreover,
\begin{gather}\label{T3.16-3}
\text{\rm Sol}(P,\vertiii{\cdot}_p)=\bigcap_{i=1}^n \text{\rm \textbf{C}}(v_i,x_i^*),
\end{gather}
where 
\begin{gather}\label{T6.1-2}
\text{\rm \textbf{C}}(v_i,x_i^*) :=\text{\rm \textbf{A}}(v_i,x_i^*)
\cap \text{\rm \textbf{D}}(v_i,x_i^*)
\end{gather}
with $\text{\rm \textbf{A}}(v_i,x_i^*)$ given by \eqref{T2.7-30} and
\begin{gather}\label{T6.1-1}
\text{\rm \textbf{D}}(v_i,x_i^*):=	
\left\{u\in X\mid \|x^*_i\|^{*q}=\dfrac{\|u-v_i\|^p}{\sum_{i=1}^{n}\|u-v_i\|^p} \right\}.
\end{gather}	
\end{theorem}

\begin{proof}
The first part  is a direct consequence of Proposition~\ref{P3.5+}\eqref{P3.6-3} and Theorem~\ref{T3.5}.
We now prove \eqref{T3.16-3}.
Let $u\in \text{\rm Sol}(P,\vertiii{\cdot}_{p})$.
In view of \eqref{P2.5-1} with $\vertiii{\cdot}:=\vertiii{\cdot}_p$, we have
\begin{gather}\label{T3.16-6}
\sum_{i=1}^{n}\ang{x^*_i,u-v_i}=\left(\sum_{i=1}^{n}\|u-v_i\|^p\right)^{\frac{1}{p}}.
\end{gather}
By \eqref{T3.16-6} and using the same arguments as in the proof of Proposition~\ref{P3.5+}\eqref{P3.6-3},
conditions \eqref{T3.16-2} are satisfied  with $u$ in place of $\bar u$.
Thus, $u\in\bigcap_{i=1}^n\text{\rm \textbf{C}}(v_i,x^*_i)$.
Conversely, if $u\in \bigcap_{i=1}^n \text{\rm \textbf{C}}(v_i,x_i^*)$, then
\begin{gather*}
\sum_{i=1}^{n}\ang{x^*_i,u-v_i}
=\sum_{i=1}^{n}\|x^*_i\|^*\cdot\|u-v_i\|
=\dfrac{\sum_{i=1}^{n}\|u-v_i\|^{1+\frac{p}{q}} }{\left(\sum_{i=1}^{n}\|u-v_i\|^p \right)^{\frac{1}{q}}}=
\left(\sum_{i=1}^{n}\|u-v_i\|^p \right)^{\frac{1}{p}}.
\end{gather*}	
By Theorem~\ref{T3.5}, we have  $u\in \text{\rm Sol}(P,\vertiii{\cdot}_p)$.
\end{proof}

\if{
\begin{remark}\label{R2.16}
Observe that $ \text{\rm \textbf{C}}(v_i,x^*_i):=\{v_i\}$ if $x^*_i=0$.
In general, there is no closed-form expression for the elements of $ \text{\rm \textbf{C}}(v_i,x^*_i)$, and
computational methods are required to find them these elements approximately.
\end{remark}	
}\fi
\begin{remark}
To the best of our knowledge, the dual necessary and sufficient optimality conditions and the construction of the solution set in Theorem~\ref{T3.16} have not been previously studied in the literature.
\end{remark}

The following examples illustrate Theorem~\ref{T3.5}.
\begin{example}\label{E7.2}
Let $\R^2$ be equipped with  some norm $\|\cdot\|$, vectors $v,w$ and $\bar u$ be given by \eqref{E4.4.1}.
The objective function \eqref{6.1} is of the form
\begin{gather}\label{E7.2-0}
f(u)=\left(\|u-v\|^p+\|u-w\|^p\right)^{\frac{1}{p}}
\end{gather}
for all $u\in\R^2$.
By Example~\ref{E3.7}, we have $\bar u\in\text{\rm Sol}(P)$.
Define $v^*:=\left(2^{-\frac{1}{q}},0\right)$ and $w^*:=\left(-2^{-\frac{1}{q}},0\right)$.
Then $v^*+w^*=0$.
\begin{enumerate}
\item\label{E7.2-1}
Suppose that $\|\cdot\|:=\|\cdot\|_\infty$.
Then
\begin{gather*}
\ang{v^*,\bar u-v}=\|v^*\|_1\cdot\|\bar u-v\|_\infty
=2^{-\frac{1}{q}},\;\;
\ang{w^*,\bar u-w}=\|w^*\|_1\cdot\|\bar u-w\|_\infty=2^{-\frac{1}{q}},\\
\|v^*\|^q_1=\dfrac{\|\bar u-v\|^p_\infty}{
\|\bar u-v\|^p_\infty+\|\bar u-w\|^p_\infty}=\frac{1}{2},\;
\|w^*\|^q_1=\dfrac{\|\bar u-w\|^p_\infty}{
	\|\bar u-v\|^p_\infty+\|\bar u-w\|^p_\infty}=\frac{1}{2}.
\end{gather*}	
The sets $\text{\rm \textbf{A}}_{\|\cdot\|_\infty}(v,v^*)$ and 	$\text{\rm \textbf{A}}_{\|\cdot\|_\infty}(w,w^*)$
are given by \eqref{E4.4-7} and \eqref{E4.4-8}.
By \eqref{T6.1-1},
\begin{align*}
\text{\rm \textbf{D}}_{\|\cdot\|_\infty}(v,v^*)
&=\left\{u\in\R^2\mid \frac{1}{2} =\dfrac{\|u-v\|_\infty^p}{\|u-v\|^p_\infty+\|w-u\|^p_\infty} \right\}\\
&=\left\{u\in\R^2\mid \|u-v\|_\infty=\|u-w\|_\infty \right\}\\
&=\left\{(u_1,u_2)\mid |u_2|\ge |u_1|,\;|u_2|\ge|u_1-2| \right\}\cup \left\{(1,\lambda)\mid |\lambda|\le 1 \right\},\\
\text{\rm \textbf{D}}_{\|\cdot\|_\infty}(w,w^*)
&=\left\{u\in\R^2\mid \frac{1}{2} =\dfrac{\|u-w\|_\infty^p}{\|u-v\|^p_\infty+\|u-w\|^p_\infty} \right\}=\text{\rm \textbf{D}}_{\|\cdot\|_\infty}(v,v^*).
\end{align*}	
By \eqref{T6.1-2},
\begin{align*}
\text{\rm \textbf{C}}_{\|\cdot\|_\infty}(v,v^*) 
&=\left\{(1,\lambda)\mid |\lambda|\le 1\right\}
\cup\left\{(|\lambda|,\lambda)\mid |\lambda|\ge 1\right\},\\
\text{\rm \textbf{C}}_{\|\cdot\|_\infty}(w,w^*) 
&=\left\{(1,\lambda)\mid |\lambda|\le 1\right\}
\cup\left\{(\lambda,|2-\lambda|)\mid \lambda\le 1\right\}.
\end{align*}	
By \eqref{T3.16-3}, we have
$ \text{\rm Sol}(P)=\text{\rm \textbf{C}}_{\|\cdot\|_\infty}(v,v^*)\cap \text{\rm \textbf{C}}_{\|\cdot\|_\infty}(w,w^*)=\left\{(1,\lambda)\mid |\lambda|\le 1\right\}.$
\begin{figure}[H]
\centering
\includegraphics[width=0.9\linewidth]{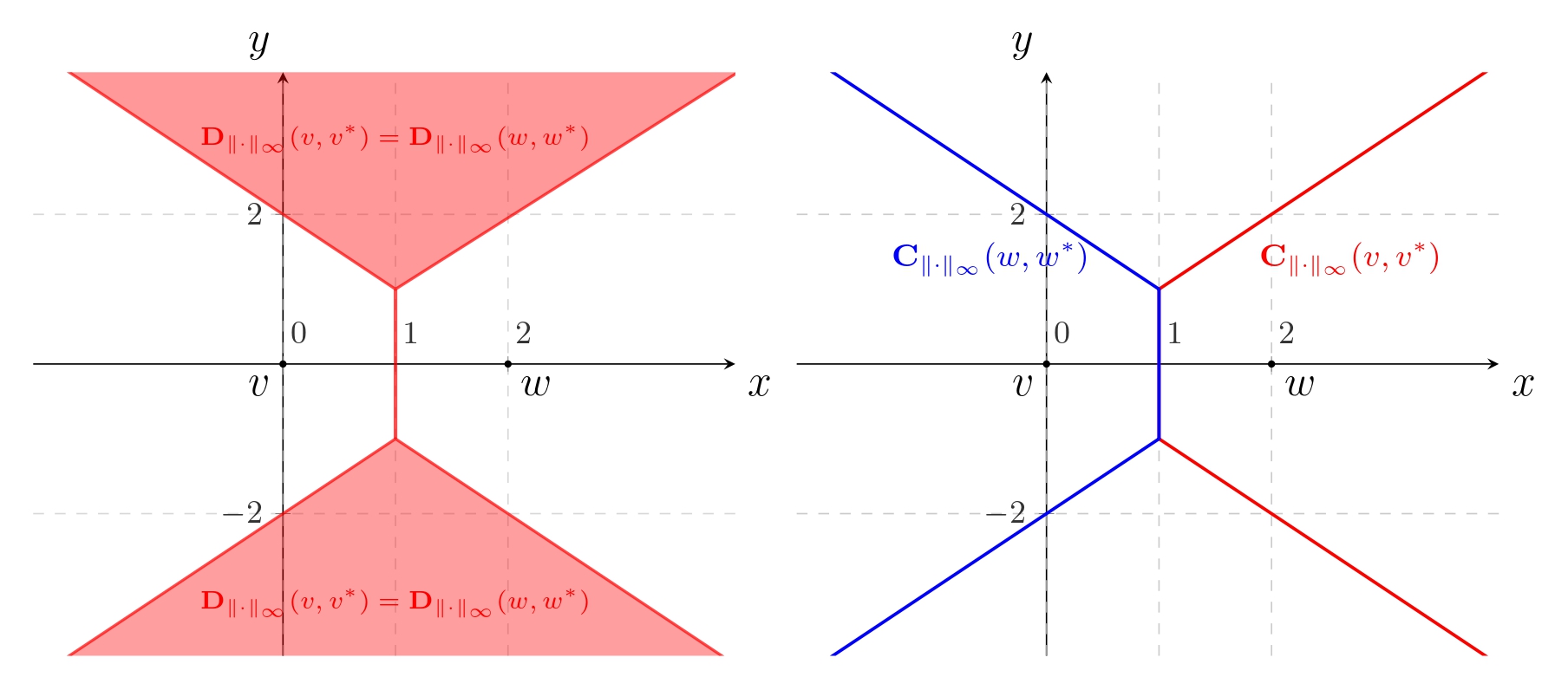}
\caption{Example~\ref{E7.2}\eqref{E7.2-1}}
\label{fig5}
\end{figure}
\item\label{E7.2-2}
Suppose that $\|\cdot\|:=\|\cdot\|_{p'}$ with $p'\in[1,\infty)$.
The dual norm  is $\|\cdot\|_{q'}$ with $q'\in(1,\infty]$ satisfying $\frac{1}{p'}+\frac{1}{q'}=1$.
We have
\begin{gather*}
\|v^*\|^q_{q'}=\dfrac{\|\bar u-v\|^p_{p'}}{
\|\bar u-v\|^p_{p'}+\|\bar u-w\|^p_{p'}}=1/2,\;
\|w^*\|^q_{q'}=\dfrac{\|\bar u-w\|^p_{p'}}{
\|\bar u-v\|^p_{p'}+\|\bar u-w\|^p_{p'}}=1/2,\\
\ang{v^*,\bar u-v}=\|v^*\|_{q'}\cdot\|\bar u-v\|_{p'}=2^{-\frac{1}{q}},\;\ang{w^*,\bar u-w}=\|w^*\|_{q'}\cdot\|\bar u-w\|_{p'}=2^{-\frac{1}{q}}.
\end{gather*}	
By  Lemma~\ref{L1.1}\eqref{L1.1-3} and \eqref{L1.1-2},
\begin{align*}
\text{\rm \textbf{A}}_{\|\cdot\|_{p'}}(v,v^*)
=\left\{(\lambda,0)\mid\lambda\ge 0 \right\},\;\;
\text{\rm \textbf{A}}_{\|\cdot\|_{p'}}(w,w^*)
=\left\{(\lambda,0)\mid \lambda\le 2\right\}.
\end{align*}	
By \eqref{T6.1-1},
\begin{align*}
\text{\rm \textbf{D}}_{\|\cdot\|_{p'}}(v,v^*)
&=\left\{u\in\R^2\mid \frac{1}{2} =\dfrac{\|u-v\|_{p'}^p}{\|u-v\|^p_{p'}+\|u-w\|^p_{p'}} \right\}\\
&=\left\{u\in\R^2\mid \|u-v\|_{p'}=\|u-w\|_{p'} \right\}
=\left\{(1,\lambda)\mid\lambda\in\R\right\},\\
\text{\rm \textbf{D}}_{\|\cdot\|_{p'}}(w,w^*)
&=\left\{u\in\R^2\mid \frac{1}{2} =\dfrac{\|u-w\|_{p'}^p}{\|u-v\|^p_{p'}+\|u-w\|^p_{p'}} \right\}
=\text{\rm \textbf{D}}_{\|\cdot\|_{p'}}(v,v^*).
\end{align*}	
By \eqref{T6.1-2},
\begin{align*}
\text{\rm \textbf{C}}_{\|\cdot\|_{p'}}(v,v^*) 
=\text{\rm \textbf{A}}_{\|\cdot\|_{p'}}(v,v^*)
\cap \text{\rm \textbf{D}}_{\|\cdot\|_{p'}}(v,v^*)=\{\bar u\},\\
\text{\rm \textbf{C}}_{\|\cdot\|_{p'}}(w,w^*) 
=\text{\rm \textbf{A}}_{\|\cdot\|_{p'}}(w,w^*)
\cap \text{\rm \textbf{D}}_{\|\cdot\|_{p'}}(w,w^*)=\{\bar u\}.
\end{align*}	
By \eqref{T3.16-3}, we have
$\text{\rm Sol}(P)=\text{\rm \textbf{C}}_{\|\cdot\|_{p'}}(v,v^*)\cap \text{\rm \textbf{C}}_{\|\cdot\|_{p'}}(w,w^*)=\{\bar u\}.$
\begin{figure}[H]
\centering
\includegraphics[width=1\linewidth]{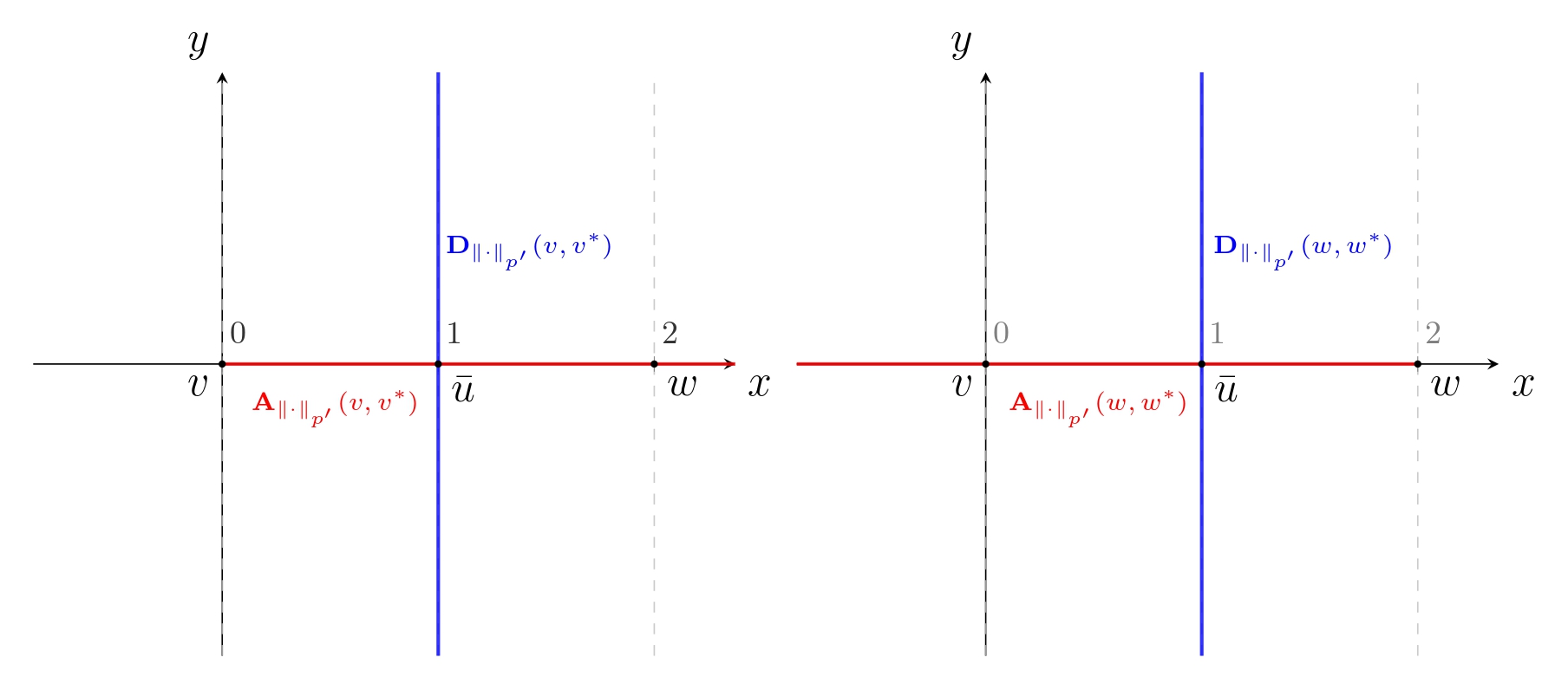}
\caption{Example~\ref{E7.2}\eqref{E7.2-2}}
\label{fig6}
\end{figure}
\end{enumerate}
\end{example}

\begin{example}
Let $X$ be the Hilbert space in Example~\ref{E4.3}, and vectors
$v,w,\bar u$ be given by \eqref{E5.4-1}.
The objective function is of the form \eqref{E7.2-0} for all $u\in X$.
By Example~\ref{E3.7}, we have $\bar u\in\text{\rm Sol}(P)$.
By Theorem~\ref{T3.16}, there exist $v^*,w^*\in X^*$ such that
\begin{gather*}
\|v^*\|^{*q}=\dfrac{\|\bar u-v\|^p}{
\|\bar u-v\|^p+\|\bar u-w\|^p},\;
\|w^*\|^{*q}=\dfrac{\|\bar u-w\|^p}{
\|\bar u-v\|^p+\|\bar u-w\|^p},\\
\|v^*\|^{*q}+\|w^*\|^{*q}=1,\;
\ang{v^*,\bar u-v}=\|v^*\|^*\cdot\|\bar u-v\|,\;\ang{w^*,\bar u-w}=\|w^*\|^*\cdot\|\bar u-w\|.
\end{gather*}	
By the Riesz representation theorem, there exist unique vectors
$\bar v:=-2^{-\frac{1}{q}}\sqrt 3 v$ and $\bar w:=2^{-\frac{1}{q}}\sqrt 3 v$  such that the above conditions are satisfied with $\bar v$ and $\bar w$ in place of $v^*$ and $w^*$, respectively.	
Indeed, 
\begin{gather*}
\|\bar v\|^{q}=
\frac{\|\bar u-v\|^p}{
\|\bar u-v\|^p+\|\bar u-w\|^p}=1/2,\;\|\bar w\|^q=
\frac{\|\bar u-w\|^p}{\|\bar u-v\|^p+\|\bar u-w\|^p}=1/2,\\
\ang{\bar v,\bar u-v}
=\|\bar v\|\cdot\|\bar u-v\|=\frac{1}{2^{\frac{1}{q}}\sqrt 3},\;\ang{\bar w,\bar u-w}
=\|\bar w\|\cdot\|\bar u-w\|=\frac{1}{2^{\frac{1}{q}}\sqrt 3}.
\end{gather*}	
By  Lemma~\ref{L1.1}\eqref{L1.1-1},
\begin{align*}
\text{\rm \textbf{A}}_{\|\cdot\|}(v,v^*)
=\left\{\lambda v\mid \lambda\le 1\right\},\;\;
\text{\rm \textbf{A}}_{\|\cdot\|}(w,w^*)
=\left\{\lambda v\mid \lambda\ge -1\right\}.
\end{align*}	
By \eqref{T6.1-1},
\begin{align*}
\text{\rm \textbf{D}}_{\|\cdot\|}(v,v^*)
&=\left\{u\in X\mid \frac{1}{2} =\dfrac{\|u-v\|^p}{\|u-v\|^p+\|u-w\|^p} \right\}\\
&=\left\{u\in X\mid \|u-v\|=\|u-w\|\right\}
=\left\{u\in X\mid \ang{u,v}=0 \right\},\\
\text{\rm \textbf{D}}_{\|\cdot\|}(w,w^*)
&=\left\{u\in X\mid \frac{1}{2} =\dfrac{\|u-w\|^p}{\|u-v\|^p+\|u-w\|^p} \right\}
=\text{\rm \textbf{D}}_{\|\cdot\|}(v,v^*).
\end{align*}	
By \eqref{T6.1-2},
\begin{align*}
\text{\rm \textbf{C}}_{\|\cdot\|}(v,v^*) 
&=\text{\rm \textbf{A}}_{\|\cdot\|}(v,v^*)
\cap \text{\rm \textbf{D}}_{\|\cdot\|}(v,v^*)=\{\bar u\},\\
\text{\rm \textbf{C}}_{\|\cdot\|}(w,w^*) 
&=\text{\rm \textbf{A}}_{\|\cdot\|}(w,w^*)
\cap \text{\rm \textbf{D}}_{\|\cdot\|}(w,w^*)=\{\bar u\}.
\end{align*}	
By \eqref{T3.16-3}, we have
$\text{\rm Sol}(P)=\text{\rm \textbf{C}}_{\|\cdot\|}(v,v^*)\cap \text{\rm \textbf{C}}_{\|\cdot\|}(w,w^*)=\{\bar u\}.$
\end{example}	

\section*{Acknowledgments}
The author wishes to thank Professor Alexander Kruger for  his comments and the anonymous referee for suggestions which helped us improve the manuscript.

\section*{Disclosure statement}
The author reports there are no competing interests to declare.

\section*{Data availability statement}
Data sharing is not applicable to this article as no new data were created or analyzed in this study.

\section*{Funding}
{This research was funded by the Postdoctoral Scholarship Programme of the Vingroup Innovation Foundation (VinIF) under grant code  VINIF.2022.STS.40.
}

\section*{ORCID}
Nguyen Duy Cuong http://orcid.org/0000-0003-2579-3601


\end{document}

%% file: macro.tex
\newcommand{\al}{\alpha}

\newcommand{\bx}{\bar x}

\newcommand {\R} {\mathbb R}

\newcommand {\dom} {{\rm dom}\,}

\newcommand {\sd} {\partial}

\newcommand{\vertiii}[1]{\left\vert\kern-0.25ex\left\vert\kern-0.25ex\left\vert #1\right\vert\kern-0.25ex\right\vert\kern-0.25ex\right\vert}
\newcommand{\vertiiiBig}[1]{\Big\vert\kern-0.25ex\Big\vert\kern-0.25ex\Big\vert #1\Big\vert\kern-0.25ex\Big\vert\kern-0.25ex\Big\vert}

\newcommand{\norm}[1]{\left\Vert#1\right\Vert}

\newcommand{\ang}[1]{\left\langle #1 \right\rangle}

\newcounter{mycount}